\documentclass[11pt]{amsart}
\input xy
\xyoption{all}

\usepackage{amsmath,amsfonts,amsthm,amssymb,latexsym,amscd,amsopn,graphics,color,enumerate,lscape,mathtools,mathrsfs,placeins}
\usepackage{hyperref}
\usepackage{tikz-cd}
\usepackage{todonotes}
\usepackage{comment}
\usepackage{enumitem}
\usepackage{rotating}
\usepackage[nocompress,noadjust]{cite}
\usepackage{stmaryrd,setspace}
\usepackage[margin=1in]{geometry}

\allowdisplaybreaks

\theoremstyle{plain}

\newtheorem{theorem}{Theorem}[section]
  \newtheorem{corollary}[theorem]{Corollary}

  \newtheorem{prop}[theorem]{Proposition}
  
  \newtheorem{lemma}[theorem]{Lemma}

\theoremstyle{definition}
  \newtheorem{defn}[theorem]{Definition}

\theoremstyle{remark}
 \newtheorem*{remark}{Remark}

  \makeatletter
\newenvironment{subtheorem}[1]{%
  \def\subtheoremcounter{#1}%
  \refstepcounter{#1}%
  \protected@edef\theparentnumber{\csname the#1\endcsname}%
  \setcounter{parentnumber}{\value{#1}}%
  \setcounter{#1}{0}%
  \expandafter\def\csname the#1\endcsname{\theparentnumber.\Alph{#1}}%
  \ignorespaces
}{%
  \setcounter{\subtheoremcounter}{\value{parentnumber}}%
  \ignorespacesafterend
}
\makeatother
\newcounter{parentnumber}

\newtheoremstyle{TheoremNum}
        {8.0pt plus 2.0pt minus 4.0pt}
        {8.0pt plus 2.0pt minus 4.0pt}
        {\itshape}
        {}
        {\bfseries}
        {.}
        {.5em}
        {\thmname{#1}\thmnote{ \bfseries #3}}
\theoremstyle{TheoremNum}

\makeatletter
\def\@tocline#1#2#3#4#5#6#7{\relax
  \ifnum #1>\c@tocdepth 
  \else
    \par \addpenalty\@secpenalty\addvspace{#2}%
    \begingroup \hyphenpenalty\@M
    \@ifempty{#4}{%
      \@tempdima\csname r@tocindent\number#1\endcsname\relax
    }{%
      \@tempdima#4\relax
    }%
    \parindent\z@ \leftskip#3\relax \advance\leftskip\@tempdima\relax
    \rightskip\@pnumwidth plus4em \parfillskip-\@pnumwidth
    #5\leavevmode\hskip-\@tempdima
      \ifcase #1
       \or\or \hskip 30pt \or \hskip 2em \else \hskip 3em \fi%
      #6\nobreak\relax
    \hfill\hbox to\@pnumwidth{\@tocpagenum{#7}}\par
    \nobreak
    \endgroup
  \fi}
\makeatother

\def\Z{{\mathbb Z}}

\def\B{{\mathcal B}}

\def\Z{{\mathbb Z}}

\def\inv{{\rm inv}}

\def\Vol{{\rm Vol}}

\def\R{{\mathbb R}}
\def\C{{\mathbb C}}
\def\F{{\mathbb F}}

\def\Q{{\mathbb Q}}

\def\Z{{\mathbb Z}}

\def\F{{\mathbb F}}
\def\Q{{\mathbb Q}}

\def\cJ{{\mathcal J}}

\def\Z{{\mathbb Z}}

\renewcommand{\lambda}{Q}

\newcommand{\nc}{\newcommand}
\nc{\on}{\operatorname}
\nc{\renc}{\renewcommand}
\nc{\wt}{\widetilde}
\nc{\defeq}{\vcentcolon=}
\nc{\eqdef}{=\vcentcolon}
\nc{\Spec}{\on{Spec}}
\nc{\ol}{\overline}
\renc{\d}{\partial}
\nc{\mc}{\mathcal}

\makeatletter
\@namedef{subjclassname@2020}{%
  \textup{2020} Mathematics Subject Classification}
\makeatother

\newcommand\rddots
   {\mathinner{\mkern1mu
       \raise 1pt\hbox{.}\mkern2mu
       \raise 4pt\hbox{.}\mkern2mu
       \raise 7pt\hbox{.}\mkern1mu}}

\title[The mean number of $2$-torsion elements in the class groups of cubic orders]{\vspace*{-0.5in}The mean number of $2$-torsion elements \\ in the class groups of cubic orders\vspace*{-0.1in}}

\author[Ashvin A.~Swaminathan]{Ashvin A.~Swaminathan$^\dagger$}
\thanks{$^\dagger$Dept.~of Mathematics, Harvard University, Cambridge, MA 02138. Email: \texttt{swaminathan@math.harvard.edu}.}

\subjclass[2020]{11R29, 11R45 (primary), 11H55, 11E76 (secondary)}

\keywords{Class groups, cubic orders, irreducibility.}

\date{\today}

\begin{document}


\maketitle

\vspace*{-0.4in}
\begin{abstract}
We determine the mean number of 2-torsion elements in class groups of cubic orders, when such orders are enumerated by discriminant. Specifically, we prove that when isomorphism classes of totally real (resp., complex) cubic orders are enumerated by discriminant, the average $2$-torsion in the class group is $1 + \frac{1}{4} \times \frac{\zeta(2)}{\zeta(4)}$ (resp., $1 + \frac{1}{2} \times \frac{\zeta(2)}{\zeta(4)}$). In particular, we find that the average $2$-torsion in the class group increases when one ranges over all orders in cubic fields instead of restricting to the subfamily of rings of integers of cubic fields, where the average $2$-torsion in the class group was first determined in work of Bhargava to be $\frac{5}{4}$ (resp., $\frac{3}{2}$).

By work of Bhargava--Varma, proving this result amounts to obtaining an asymptotic count of the number of ``reducible'' $\operatorname{SL}_3(\Z)$-orbits on the space $\Z^2 \otimes_{\Z} \on{Sym}^2 \Z^3$ of $3 \times 3$ symmetric integer matrices having bounded invariants and satisfying local conditions. In this paper, we resolve the generalization of this orbit-counting problem where the dimension $3$ is replaced by any fixed odd integer $N \geq 3$. More precisely, we determine asymptotic formulas for the number of reducible $\operatorname{SL}_N(\Z)$-orbits on $\Z^2 \otimes_{\Z} \on{Sym}^2 \Z^N$ satisfying general infinite sets of congruence conditions.
\end{abstract}

\tableofcontents

\section{Introduction} \label{sec-intro}

\subsection{Main results} \label{sec-remoaner}
A striking result of Bhargava's thesis was the determination of the average size of the $2$-torsion subgroup in the class groups of cubic number fields, enumerated by discriminant (see~\cite[Theorem~5.4]{bhargthesis} and~\cite[Theorem~5]{MR2183288}). Specifically, Bhargava proved that when totally real (resp., complex) cubic fields are enumerated by the absolute values of their discriminants, the average $2$-torsion in the class group is equal to $\frac{5}{4}$ (resp., $\frac{3}{2}$). This result remains one of just a handful of cases of the heuristics on class groups of number fields, formulated in the foundational works of Cohen--Lenstra~\cite{MR756082}, Cohen--Martinet~\cite{MR866103}, and Malle~\cite{MR2778658}, that have ever been proven.

The main result of this paper constitutes a generalization of Bhargava's breakthrough to the full family of all \emph{orders} in cubic fields. We prove:

\begin{subtheorem}{theorem}
\begin{theorem} \label{thm-main2tors}
When irreducible cubic orders $\mc{O}$ over $\Z$ are enumerated by the absolute values of their discriminants, the average size of $\on{Cl}(\mc{O})[2]$ is:
\begin{itemize}[leftmargin=2em]
\vspace*{3pt}
\item[{\rm (a)}] $\displaystyle 1 + \frac{1}{4} \times \frac{\zeta(2)}{\zeta(4)}$ for the family of totally real cubic orders $\mc{O}$, and
\vspace*{6pt}
\item[{\rm (b)}] $\displaystyle 1 + \frac{1}{2} \times \frac{\zeta(2)}{\zeta(4)}$ for the family of complex cubic orders $\mc{O}$. \vspace{1pt}
\end{itemize}
\end{theorem}

\begin{remark}
Note that $\frac{\zeta(2)}{\zeta(4)} = \frac{15}{\pi^2} \approx 1.51982 > 1$.
\end{remark}

Let $N = 2n+1 \geq 3$ be an odd integer, and let $U_N$ be the affine space over $\Z$ whose $R$-points are given by binary $N$-ic forms over $R$ for any $\Z$-algebra $R$. When $N = 3$, the Delone--Faddeev--Levi correspondence~\cite{MR0160744} states that the map sending an irreducible binary cubic form $f \in U_3(\Z)$ to the ring $R_f$ of global sections of the subscheme of $\mathbb{P}^1_{\Z}$ cut out by $f$ defines a bijection between the irreducible orbits of $\on{GL}_2(\Z)$ on $U_3(\Z)$ and the set of isomorphism classes of orders in cubic number fields. In light of this, studying cubic orders amounts to studying irreducible integral binary cubic forms that lie in a fundamental region $\mc{F}$ for the action of $\on{GL}_2(\Z)$ on $U_3(\R)$. 

The region $\mc{F}$ may be chosen so that the set of $f \in \mc{F}$ with discriminant $\on{disc}(f) \ll X$ is approximately the set of $f \in \mc{F}$ whose coefficients are all $\ll X^{\frac{1}{4}}$. Thus, we expect that the averages in Theorem~\ref{thm-main2tors} should remain the same if we replace the family of cubic orders, enumerated by discriminant, with the family of cubic orders of the form $R_f$,\footnote{Note that each cubic order $\mc{O}$ in this family occurs infinitely many times, once for every $f$ such that $\mc{O} \simeq R_f$!} where $f$ runs through irreducible integral binary cubic forms enumerated by \emph{height}; here we define the height $\on{H}(f)$ of a binary form to be the maximum of the absolute values of its coefficients. Indeed, by modifying the proof of Theorem~\ref{thm-main2tors}, we obtain the following variant:

\begin{theorem} \label{thm-main2tors2}
When irreducible binary cubic forms $f \in U_3(\Z)$ are enumerated by height, the average size of $\on{Cl}(R_f)[2]$ is:
\begin{itemize}[leftmargin=2em]
\vspace*{3pt}
\item[{\rm (a)}] $\displaystyle 1 + \frac{1}{4} \times \frac{\zeta(2)}{\zeta(4)}$ for the family of totally real binary cubic forms $f$, and
\vspace*{6pt}
\item[{\rm (b)}] $\displaystyle 1 + \frac{1}{2} \times  \frac{\zeta(2)}{\zeta(4)}$ for the family of complex binary cubic forms $f$. \vspace{1pt}
\end{itemize}
\end{theorem}
\end{subtheorem}

To prove Theorems~\ref{thm-main2tors} and~\ref{thm-main2tors2}, we utilize an orbit parametrization, discovered by Bhargava when $N = 3$ (see~\cite[Theorem~4]{MR2081442}), and generalized to all odd $N$ by Wood (see~\cite[Theorem~1.3]{MR3187931}). Fix an irreducible binary $N$-ic form $f \in U_N(\Z)$, and assume $f$ is primitive if $N > 3$. Let $K_f \defeq \Q \otimes_{\Z} R_f$, and let $W_N$ denote the affine space over $\Z$ whose $R$-points are given by pairs of symmetric $N \times N$ matrices over $R$. Then the Bhargava--Wood parametrization takes as input a pair $(I,\delta)$, where $I$ is a $2$-torsion ideal class of the order $R_f$ and $\delta \in K_f^\times/K_f^{\times2}$ is a generator of $I^2$ having square norm, and it produces as output the $\on{SL}_N(\Z)$-orbit of a pair $(A,B) \in W_N(R)$ such that $\on{inv}(A,B) \defeq (-1)^{n} \det(xA - yB) = f(x,y)$. 

The set of pairs $(I,\delta)$ corresponding to a form $f \in U_N(\Z)$ naturally partitions into two subsets depending on whether or not $\delta \equiv 1 \in K_f^\times/K_f^{\times2}$. Via the parametrization, pairs $(I,\delta)$ with $\delta \not\equiv 1$ (resp., $\delta \equiv 1$) correspond to so-called \emph{irreducible} (resp., \emph{reducible}) $\on{SL}_N(\Z)$-orbits on $W_N(\Z)$. Here, (the $\on{SL}_N(\Z)$-orbit of) a pair $(A,B) \in W_N(\Z)$ is said to be reducible if, when $A$ and $B$ are viewed as symmetric bilinear forms over $\Q$, they share an isotropic space over $\Q$ of maximal dimension. Geometrically, $A$ and $B$ may be viewed as defining a pair of quadric hypersurfaces in $\mathbb{P}^{N-1}$, and the condition that $(A,B)$ is reducible is equivalent to stipulating that the (finite) Fano scheme parametrizing maximal linear spaces contained in the intersection of these two quadrics has a $\Q$-rational point.

Note that if $(I,\delta)$ is a pair with $\delta \equiv 1$, then $I$ is the class of a \emph{$2$-torsion ideal} of $R_f$ (i.e., a fractional ideal of $R_f$ that squares to the unit ideal). Thus, the set of reducible $\on{SL}_N(\Z)$-orbits of pairs $(A,B) \in \on{inv}^{-1}(f) \cap W_N(\Z)$ is in bijection with the $2$-torsion subgroup of the ideal group $\mc{I}(R_f)$. If $R_f$ is the maximal order in $K_f$ (i.e., $R_f$ is integrally closed in $K_f$), then $\mc{I}(R_f)$ is torsion-free, and the only $2$-torsion ideal is the trivial one. Thus, there is exactly one reducible orbit corresponding to $f$ via the parametrization, and the problem of determining the average size of the $2$-torsion in the class groups of maximal orders $R_f$ amounts to counting just the irreducible orbits. A systematic method for counting irreducible orbits of representations (where the definition of ``irreducible'' depends on the representation) was developed by Bhargava in his thesis (see, e.g.,~\cite{MR2745272}) and was later vastly streamlined in the seminal work of Bhargava and Shankar on Selmer groups of elliptic curves (see, e.g.,~\cite{MR3272925}). Using this method to count just the irreducible orbits, Bhargava and Varma proved the following precursor to Theorem~\ref{thm-main2tors}:

\begin{subtheorem}{theorem}
\begin{theorem}[\protect{\cite[Theorem~2]{MR3369305}}] \label{thm-inc}
When irreducible cubic orders $\mc{O}$ are enumerated by the absolute values of their discriminants, the average size of $\on{Cl}(\mc{O})[2]$ is:
\begin{itemize}[leftmargin=2em]
\vspace*{3pt}
\item[{\rm (a)}] $\displaystyle 1 + \frac{1}{4} \times \underset{\on{disc}(\mc{O}) > 0}{\on{Avg}} \#\mc{I}(\mc{O})[2]\,$ for the family of totally real cubic orders $\mc{O}$, and
\vspace*{6pt}
\item[{\rm (b)}] $\displaystyle 1 + \frac{1}{2} \times \underset{\on{disc}(\mc{O}) < 0}{\on{Avg}}\,  \#\mc{I}(\mc{O})[2]\,$ for the family of complex cubic orders $\mc{O}$. \vspace{1pt}
\end{itemize}
\end{theorem}

Just as Theorem~\ref{thm-inc} was a precursor to Theorem~\ref{thm-main2tors}, we have the following result of Ho, Shankar, and Varma, which holds for \emph{any} $N$ and was a precursor to Theorem~\ref{thm-main2tors2} in the case $N = 3$:

\begin{theorem}[\protect{\cite[Theorem~6]{MR3782066}}] \label{thm-inc2}
Let $U_N(\Z)^{(r)}$ be the set of irreducible forms $f \in U_N(\Z)$ having $r$ real roots and $\frac{N-r}{2}$ pairs of complex roots. When forms $f \in U_N(\Z)^{(r)}$, primitive if $N > 3$, are enumerated by height, the average size of $\on{Cl}(R_f)[2]$ is $$\displaystyle 1 + 2^{1-\frac{N+r}{2} } \times \underset{\substack{f \in U_N(\Z)^{(r)} \\ {\tiny f \text{ prim.~if }N > 3}}}{\on{Avg}} \#\mc{I}(R_f)[2].$$
\end{theorem}
\end{subtheorem}

To determine precise numerical values of the average $2$-torsion in the class group for the families of orders considered in Theorems~\ref{thm-inc} and~\ref{thm-inc2}, one needs to determine the average $2$-torsion in their ideal groups. Via the Bhargava--Wood parametrization, this amounts to counting reducible $\on{SL}_N(\Z)$-orbits on $W_N(\Z)$. Systematic methods for counting reducible orbits were not available until recently, when the author, in joint work with Shankar, Siad, and Varma, developed a new technique that applies to reducible orbits of many representations of importance in arithmetic statistics~\cite{cuspy}. 

In the context of the action of $\on{SL}_N$ on $W_N$, our new technique proceeds according to the following series of steps. First, we prove in \S\ref{sec-pfthmain2} (see Proposition~\ref{prop-sasymp}) that the count of reducible $\on{SL}_N(\Z)$-orbits on $W_N(\Z)$ lying above irreducible binary forms is the same as the corresponding count for $G_N(\Z)$-orbits on $W_N^0(\Z)$, where $G_N \subset \on{SL}_N$ is a certain closed subgroup and $W_N^0 \subset W_N$ is a certain linear subspace all of whose $\Q$-points are reducible (see \S\S\ref{sec-thesub}--\ref{sec-thesubway} for the precise definitions of $G_N$ and $W_N^0$). Second, we prove that the action of $G_N$ on $W_N^0$ satisfies the following strong local-to-global principle:
\begin{theorem} \label{thm-stronglocglob}
Let $f \in U_N(\Z)$ be a binary $N$-ic form with nonzero discriminant. For each prime $p$, choose $(A_p,B_p) \in \on{inv}^{-1}(f) \cap W_N^0(\Z_p)$. Then there exists $(A,B) \in W_N^0(\Z)$, unique up to the action of $G_N(\Z)$, such that $(A,B)$ is $G_N(\Z_p)$-equivalent to $(A_p,B_p)$ for each prime $p$.
\end{theorem}
To prove Theorem~\ref{thm-stronglocglob}, we rely on~\cite[Theorem~22]{cuspy}, which is a general result that provides criteria under which the integral orbits of a representation satisfy a local-to-global principle. Theorem~\ref{thm-stronglocglob} follows by verifying that the representation of $G_N$ on $W_N^0$ satisfies these criteria.

As a consequence of Theorem~\ref{thm-stronglocglob}, counting $G_N(\Z)$-orbits on $W_N^0(\Z)$ amounts to counting $G_N(\Z_p)$-orbits on $W_N^0(\Z_p)$ for every prime $p$. In \S\ref{sec-themain2} we combine this local-to-global principle with results of Bhargava--Shankar--Wang~\cite{sqfrval2} to obtain an asymptotic for the count of reducible $\on{SL}_N(\Z)$-orbits on $W_N(\Z)$ satisfying certain infinite families of congruence conditions, which we define precisely as folows:

\begin{defn}
We call a subset $\mathfrak{S} \subset W_N^0(\Z)$ a \emph{big family in $W_N^0(\Z)$} if $\mathfrak{S} = \on{inv}^{-1}(U_N(\Z)^{(r)}) \cap W_N^0(\Z) \cap \bigcap_p \mathfrak{S}_p$, where the sets $\mathfrak{S}_p \subset W_N^0(\Z_p)$ satisfy the \mbox{following properties:}
\begin{enumerate}[leftmargin=20pt]
\item[(1)] $\mathfrak{S}_p$ is $G_N(\Z_p)$-invariant and is the preimage under reduction modulo $p^j$ of a nonempty subset of $W_N^0(\Z/p^j\Z)$ for some $j > 0$ for each $p$; and
\item[(2)] For all sufficiently large $p$, $\mathfrak{S}_p$ contains all elements $(A,B) \in W_N^0(\Z_p)$ such that $\lambda(A,B)$ is a $p$-adic unit, where $\lambda$ is a certain relative invariant for the action of $G_N$ on $W_N^0$ over $\Z$ (see Proposition~\ref{prop-castle} for the precise definition of $Q$).
\end{enumerate}
We call a subset $S \subset W_N(\Z)$ a \emph{big family in $W_N(\Z)$} if $S = \on{inv}^{-1}(U_N(\Z)^{(r)}) \cap \bigcap_p S_p$, where the sets $S_p \subset W_N(\Z_p)$ are $\on{SL}_N(\Z_p)$-invariant, and the intersections $\mathfrak{S}_p = S_p \cap W_N^0(\Z_p)$ define a big family $\mathfrak{S} = \on{inv}^{-1}(U_N(\Z)^{(r)}) \cap W_N^0(\Z) \cap \bigcap_p \mathfrak{S}_p$ in $W_N^0(\Z)$.
\end{defn}
Define the \emph{height} of $(A,B)$ by $\on{H}(A,B) \defeq \on{H}(\on{inv}(A,B))$. The following result gives an asymptotic for the count of reducible $\on{SL}_N(\Z)$-orbits of bounded height in a big family in $W_N(\Z)$:
\begin{theorem} \label{thm-main2}
Let $X > 0$, and let $\on{N}_N^{(r)}(X)$ be the number of binary forms in $U_N(\Z)^{(r)}$ whose coefficients are of size at most $X$. Let $S \subset \on{inv}^{-1}(U_N(\Z)^{(r)})$ be a big family in $W_N(\Z)$. Then the number of reducible $\on{SL}_N(\Z)$-orbits on $S$ of height less than $X$ is given by
$$\on{N}_N^{(r)}(X) \times \prod_p \int_{f \in U_N(\Z_p)} \#\left(\frac{\on{inv}^{-1}(f) \cap S_{p} \cap W_N^0(\Z_p)}{G_N(\Z_p)}\right)df + o(X^{N+1}),$$
where $df$ is the Haar measure on $U_N(\Z_p)$, normalized so that $\on{Vol}(U_N(\Z_p)) = 1$.
\end{theorem}

Theorem~\ref{thm-main2} expresses the asymptotic count of reducible $\on{SL}_N(\Z)$-orbits on $S$ of bounded height in terms of a product of local integrals. The integrand of the integral at $p$ --- namely, the number of $G_N(\Z_p)$-orbits of pairs $(A,B) \in S_p \cap W_N^0(\Z_p)$ with $\on{inv}(A,B) = f$ --- appears to be quite difficult to evaluate in general, even for small degrees $N$ and simple choices of the sets $S_p$. On the other hand, the integral at $p$ can be rendered more tractable by performing a suitable change-of-variables, where instead of integrating over $f \in U_N(\Z_p)$, one integrates over $(A,B) \in W_N^0(\Z_p)$. Upon performing this change-of-variables at each prime $p$, we obtain the following variant of Theorem~\ref{thm-main2}:

\begin{theorem} \label{thm-main3}
Let $X > 0$, and let $S \subset \on{inv}^{-1}(U_N(\Z)^{(r)})$ be a big family in $W_N(\Z)$. Then the number of reducible $\on{SL}_N(\Z)$-orbits on $S$ of height less than $X$ is given by
$$\on{N}_N^{(r)}(X) \times \prod_p \xi_{p,n} \int_{w \in S_p \cap W_N^0(\Z_p)} |\lambda(w)|_p\, dw + o(X^{N+1}),$$
where $\xi_{p,n} \defeq (1-p^{-1})^{-1}(1 - p^{-n-1})^{-1} \times \prod_{i = 2}^{n} (1 - p^{-i})^{-1}$, $|-|_p$ denotes the usual $p$-adic absolute value, and where $dw$ denotes the Haar measure on $W_N^0(\Z_p)$, normalized so that $\on{Vol}(W_N^0(\Z_p)) = 1$.
\end{theorem}

The formulation of the asymptotic given in Theorem~\ref{thm-main3} is far more conducive to evaluation in specific examples. For instance, taking $S_p = W_N(\Z_p)$ for every prime $p$, we obtain the following asymptotic formula for the total count of reducible $\on{SL}_N(\Z)$-orbits on $W_N(\Z)$:

\begin{theorem} \label{thm-main1}
Let $X > 0$. Then the number of reducible $\on{SL}_N(\Z)$-orbits on $\on{inv}^{-1}(U_N(\Z)^{(r)})$ of height less than $X$ is given by
$$\on{N}_N^{(r)}(X) \times \prod_{i = 2}^N \zeta(i) + o(X^{N+1}).$$
\end{theorem}

Strikingly, Theorem~\ref{thm-main1} implies that the average number of reducible orbits lying above an integral binary $N$-ic form is equal to $\prod_{i = 2}^N \zeta(i)$, which is simply the fundamental volume of the group $\on{SL}_N$ (i.e., the volume of $\on{SL}_N(\Z) \backslash \on{SL}_N(\R)$ with respect to Haar measure on $\on{SL}_N(\R)$, suitably normalized). This phenomenon --- that the average number of reducible orbits lying above a given set of invariants equals the fundamental volume of the group --- holds for many other representations of interest in arithmetic statistics. Indeed, as proven in~\cite[Theorem~1]{cuspy}, this holds for the representation of the split orthogonal group on the space of $N$-ary quadratic forms for every integer $N \geq 3$, odd or even; furthermore, as explained in~\cite[Question 2 and subsequent discussion]{cuspy}, this also holds for the representations of $\on{GL}_2$ on binary cubic and quartic forms.

We now turn our attention to applying Theorem~\ref{thm-main3} to prove our main results on class group statistics, namely Theorems~\ref{thm-main2tors} and~\ref{thm-main2tors2}. We say that a pair $(A,B) \in W_N(\Z_p)$ is \emph{projective} if it corresponds to a $2$-torsion ideal class under the Bhargava--Wood parametrization. As explained in Definition~\ref{def-proj} (see \S\ref{sec-pf2tors}, to follow), the condition of projectivity is given by a congruence condition modulo $p$ for each prime $p$. 

It is possible to determine the proportion of projective elements of $W_N(\Z_p)$ for each prime $p$; this was essentially achieved in~\cite[\S6]{MR3369305} for $N = 3$ and~\cite[\S6.1]{MR3782066} for $N \geq 3$. On the other hand, to apply Theorem~\ref{thm-main3} it is necessary to determine the proportion of projective elements of $W_N^0(\Z_p)$ \emph{having a specified $Q$-invariant}, and this appears to be intractable in general. Nonetheless, when $N = 3$, the condition of being projective is not too complicated, and the relevant $p$-adic integrals can be evaluated. We thus obtain the following result giving the average $2$-torsion in the ideal groups of cubic orders enumerated by discriminant:

\begin{subtheorem}{theorem}
\begin{theorem} \label{thm-z2z3}
When either totally real or complex irreducible cubic orders $\mc{O}$ are enumerated by the absolute values of their discriminants, the average $2$-torsion in the ideal group is $\frac{\zeta(2)}{\zeta(4)}$.
\end{theorem}

We also simultaneously obtain the following variant of Theorem~\ref{thm-z2z3} for the family of cubic orders defined by binary cubic forms enumerated by height:

\begin{theorem} \label{thm-z2z32}
When forms $f \in U_3(\Z)^{(r)}$ are enumerated by height, the average $2$-torsion in the ideal group is $\frac{\zeta(2)}{\zeta(4)}$.
\end{theorem}
\end{subtheorem}

Theorem~\ref{thm-main2tors} (resp., Theorem~\ref{thm-main2tors2}) now follows immediately upon combining Theorem~\ref{thm-inc} (resp., Theorem~\ref{thm-inc2} in the case $N = 3$) with Theorem~\ref{thm-z2z3} (resp., Theorem~\ref{thm-z2z32}). As for orders defined by binary forms of higher odd degree, we obtain the following weaker result by simply combining Theorem~\ref{thm-main1} with Theorem~\ref{thm-inc2}:

\begin{theorem}
Let $N > 3$. When primitive forms $f \in U_N(\Z)^{(r)}$ are enumerated by height, the average sizes of $\mc{I}(R_f)[2]$ and $\on{Cl}(R_f)[2]$ are bounded.
\end{theorem}

\subsection{Historical context and related work}

The problem of studying the arithmetic statistics of non-maximal orders in number fields dates back to the seminal work of Davenport~\cite{MR43822,MR1574296}, who determined the density of discriminants of orders in cubic fields (i.e., he obtained an asymptotic formula for the number of cubic orders having bounded discriminant). This result was generalized in work of Bhargava, Shankar, and Tsimerman~\cite[Theorem~8]{MR3090184}, who determined the density of discriminants of cubic orders satisfying general infinite sets of local specifications. As for orders of higher degree, Bhargava used the parametrizations of quartic/quintic orders that he developed in his thesis~\cite{MR2113024,MR2373152} to determine asymptotic formulas for the number of orders having bounded discriminant in quartic/quintic fields with Galois group equal to the full symmetric group~\cite{MR2183288,MR2745272}.

More recently, progress has been made toward understanding the distribution of \emph{class groups} of orders in number fields. The first result in this direction is due to Bhargava and Varma~\cite{MR3471250}, who determined the average $3$-torsion in the class groups of quadratic orders. Specifically, they showed that when real (resp., complex) quadratic orders are enumerated by discriminant, the average $3$-torsion in the class group is given by $1 + \frac{1}{3} \times \frac{\zeta(2)}{\zeta(3)}$ (resp., $1 + \frac{\zeta(2)}{\zeta(3)}$). This extends an earlier result of Davenport and Heilbronn~\cite{MR491593}, who determined the corresponding average to be $\frac{4}{3}$ (resp., $2$) for the family of maximal quadratic orders (i.e., rings of integers of quadratic number fields).

As explained in \S\ref{sec-remoaner}, Bhargava and Varma proved Theorem~\ref{thm-inc}, which describes the average $2$-torsion in the class groups of cubic orders in terms of the average $2$-torsion in their ideal groups~\cite{MR3369305}, and their result was generalized to orders defined by binary forms of any odd degree $N \geq 3$ by Ho, Shankar, and Varma, who proved Theorem~\ref{thm-inc2}. A similar result was proven for \emph{monogenic} orders --- i.e., orders defined by \emph{monic} binary forms --- of odd degree by Siad in~\cite[Theorem~9]{Siadthesis1}. In the case of monogenic orders of odd degree $N$, the problem of counting $2$-torsion ideals boils down to a problem of counting reducible orbits of the aforementioned representation of the split orthogonal group acting on the space of $N$-ary quadratic forms. As explained in \S\ref{sec-remoaner}, asymptotics akin to Theorems~\ref{thm-main2}--\ref{thm-main1} were proven for reducible orbits of this representation in~\cite{cuspy}; these asymptotics were then applied in~\cite[\S5.6]{swathesis} with $N = 3$ to determine the average $2$-torsion in the ideal groups of monogenic cubic orders. Combining this with the aforementioned result of Siad, we deduced (see~\cite[Theorem~173]{swathesis}) that when monogenic cubic orders are enumerated by height, the average $2$-torsion in the class group is given by $\frac{5}{4} + \frac{1}{4} \times \frac{\zeta(2)}{\zeta(3)}$ (resp., $\frac{3}{2} + \frac{1}{2} \times \frac{\zeta(2)}{\zeta(3)}$). This extends an earlier result of Bhargava, Hanke, and Shankar~\cite{BSHpreprint}, who determined the corresponding average to be $\frac{3}{2}$ (resp., $2$) for the family maximal monogenic cubic orders.

We note that the problem of counting reducible $\on{SL}_N(\Z)$-orbits on $W_N(\Z)$ was first considered by Bhargava, Shankar, and Wang in~\cite{sqfrval2}, where they used geometry-of-numbers methods to determine upper bounds of roughly the correct order of magnitude on the number of reducible orbits having large $Q$-invariant. They applied these upper bounds to prove a squarefree sieve for binary forms. In this paper, their bounds serve as a key \mbox{ingredient in the proofs of our} main theorems. Thus, we indirectly use \emph{upper bounds} for the count of reducible orbits to obtain \emph{precise asymptotics} for this count!

It is natural to ask what can be said for binary forms of even degree. We expect that the methods of~\cite{Siadthesis2,Swpreprint} can be adapted to determine the average $2$-torsion in the class group of orders defined by binary forms of even degree in terms of the average $2$-torsion in the ideal group; however, in the even-degree setting, the parametrization of $2$-torsion ideal classes is significantly more complicated (e.g., it depends on the leading coefficient of the binary forms under consideration, and one must count reducible orbits in multiple families of representations). It is also natural to ask whether analogues of Theorems~\ref{thm-main2}--\ref{thm-main1} can be proven for the action of $\on{SL}_{N}$ on pairs of $N \times N$ symmetric matrices, where $N = 2n \geq 4$ is even. This representation was studied by Bhargava in~\cite{thesource} (cf. the closely related work of Bhargava, Gross, and Wang~\cite{MR3600041}), where he shows that this representation does \emph{not} possess an analogous notion of reducibility and is thus able to determine asymptotics for the count of its integral orbits.

\section{Algebraic preliminaries}

In this section, we introduce the representation of $\on{SL}_N$ on $W_N$ and prove several useful results about the action of various subgroups of $\on{SL}_N$ on various linear subspaces in $W_N$. In particular, we define and study the subgroup $G_N$ and the linear subspace $W_N^0$ referenced in \S\ref{sec-remoaner}. We conclude by proving Theorem~\ref{thm-stronglocglob}, which gives a local-to-global principle for the action of $G_N$ on $W_N^0$.

\subsection{Action of $\on{SL}_N$ on $W_N$} \label{sec-repert}

For integers $m,m' > 0$, let $\on{Mat}_{m \times m'}$ denote the affine scheme over $\Z$ whose $R$-points are
given by the set of $m \times m'$ matrices with entries in $R$ for any $\Z$-algebra $R$. As in \S\ref{sec-remoaner}, let $N = 2n+1 \geq 3$ be an odd integer, and let $W_N$ be the space of pairs of $N \times N$ symmetric matrices (i.e., we have $W_N(R) = R^2 \otimes_R \on{Sym}^2 R^N$ for any $\Z$-algebra $R$). The space $W_N$ has a natural structure of $\on{SL}_N$-representation given as follows: for any $g \in \on{SL}_N(R)$ and any $(A,B) \in W_N(R)$, let
$$g \cdot (A,B) = \big(g A  g^T, g B g^T\big) \in W_N(R),$$
where for a matrix $M$, we denote by $M^T$ its transpose. 

Let $U_N$ be the affine scheme over $\Z$ whose $R$-points are binary $N$-ic forms over $R$ for any $\Z$-algebra $R$. Define a map $\on{inv} \colon W_N \to U_N$ as follows: given a pair $(A,B) \in W_N(R)$, we set
$$\on{inv}(A,B) \defeq (-1)^{n} \det(x A-y B) \in U_N(R).$$
One readily checks that the coefficients of $\on{inv}(A,B)$ are $\on{SL}_N$-invariant; in fact, it is known~\cite{MR516601} that if $f_i \colon W_N \to \mathbb{A}^1$ is the map that takes $(A,B) \in W_N(R)$ and returns the $x^{N-i}y^i$-coefficient of $\on{inv}(A,B)$, then the ring of polynomial invariants for the action of $\on{SL}_N$ on $W_N$ is freely generated by the functions $f_i$ for $i \in \{0, \dots, N\}$. In~\cite[\S4.1]{MR3782066}, an explicit algebraic section $\sigma_0$ defined over $\Z$ was constructed for the map $\on{inv} \colon W_N \to U_N$ (by ``algebraic'' and ``defined over $\Z$,'' we mean that the matrix entries of $\sigma_0$ are polynomials with integer coefficients). Concretely, this section takes $f(x,y) = \sum_{i = 0}^N f_ix^{N-i}y^i \in U_N(R)$ to the pair
\normalsize
$$\sigma_0(f) = \scriptstyle\left(\left[\begin{array}{ccc|cccccc}
& & & & & & 1 \\
& & & & & & \\
& & & & & \rddots & \\
& & & & 1 &  & \\ \hline
& & & f_0 & & &  \\
& & 1 & & f_2 &  &  \\
& \rddots & & & & \ddots &   \\
1 & & & & & & f_{N-1} \end{array}\right], \left[\begin{array}{cccc|ccccc}
& & & & & & 1 & \\
& & & & & & & \\
& & & & & \rddots & & \\
& & & & 1 & & & \\ \hline
& & & 1 & -f_1 & & & \\
& & \rddots & & & \ddots & & \\
1 & & & & & & -f_{N-2} & \\ & & & & & & & -f_N \end{array}\right]\right)$$
\normalsize
where empty entries are used to denote zeros, and where we have inserted horizontal and vertical lines immediately after row and column $n$ in both matrices.

Let $K$ be a field. As mentioned in \S\ref{sec-remoaner}, a pair $(A,B) \in W_N(K)$ is said to be \emph{reducible} over $K$ if the symmetric bilinear forms defined by $A$ and $B$ share a maximal (i.e., $n$-dimensional) isotropic space, and \emph{irreducible} otherwise. The condition of being (ir)reducible is evidently $\on{SL}_N(K)$-invariant, so we may speak of the (ir)reducible $\on{SL}_N(K)$-orbits on $W_N(K)$. If $R$ is an integral domain with field of fractions $K$, we say that an $\on{SL}_N(R)$-orbit on $W_N(R)$ is (ir)reducible if the corresponding property holds for the $\on{SL}_N(K)$-orbit containing it.

In the case $N = 3$, we will have occasion to combine the action of $\on{SL}_3$ on $W_3$ with the action of $\on{GL}_2$ on $U_3$. Recall that $\on{GL}_2$ acts on $U_3$ via linear change-of-variable; i.e., given $\gamma \in \on{GL}_2(R)$ and $f \in U_3(R)$, we have $(\gamma \cdot f)(x,y) = f((x,y) \cdot \gamma)$. The ring of polynomial invariants for the action of $\on{GL}_2$ on $U_3$ is freely generated by a single element known as the \emph{discriminant}. We have an action of $\on{GL}_2 \times \on{SL}_3$ on $W_3$ defined as follows: given $\gamma = \left[\begin{smallmatrix} a & b \\ c & d \end{smallmatrix}\right] \in \on{GL}_2(R)$, $g \in \on{SL}_3(R)$, and $(A,B) \in W_N(R)$, we set
$$(\gamma,g) \cdot (A,B) = \big(g(aA-bB)g^T,g(cA-dB)g^T
\big).$$
The map $\on{inv} \colon W_3 \to U_3$ is equivariant for the $\on{GL}_2$-action, and the ring of polynomial invariants for the action of $\on{GL}_2 \times \on{SL}_3$ on $W_3$ is generated by the function obtained by postcomposing the map $\on{inv}$ with the discriminant map $U_3 \to \mathbb{A}^1$.

\subsection{Action of subgroups of $\on{SL}_N$ on linear subspaces in $W_N$}

For a matrix $M$, denote its row-$i$, column-$j$ entry by $M_{ij}$. We now introduce three linear subspaces in $W_N$ and study the actions of certain subgroups of $\on{SL}_N$ on them.

\subsubsection{Action of $\on{SL}_n \times \on{SL}_{n+1}$ on $W_N^{\on{top}}$} \label{sec-221}

Let $W_N^{\on{top}} \subset W_N$ denote the linear subspace whose $R$-points are given by
$$W_N^{\on{top}}(R) \defeq \big\{(A,B) \in W_N(R) : A_{ij} = B_{ij} = 0 \text{ if $i,j \leq n$ or $i,j \geq n+1$}\big\}$$
for any $\Z$-algebra $R$. When it is convenient, we think of elements of $W_N^{\on{top}}$ as being pairs of $n \times (n+1)$ matrices, by considering just the top-right $n \times (n+1)$ blocks. 

Let $\on{SL}_n \times \on{SL}_{n+1} \subset \on{SL}_N$ be the subgroup consisting of block-diagonal matrices, where the diagonal consists of one $n \times n$ block followed by one $(n+1) \times (n+1)$ block, both having determinant $1$. The action of $\on{SL}_N$ on $W_N$ restricts to an action of $\on{SL}_n \times \on{SL}_{n+1}$ on $W_N^{\on{top}}$.

We now describe the action of $\on{SL}_n \times \on{SL}_{n+1}$ on $W_N^{\on{top}}$ over a field $K$:
\begin{prop} \label{prop-castle}
  The ring of polynomial invariants of the action of $\on{SL}_n \times \on{SL}_{n+1}$ on $W_N^{\on{top}}$ is generated by a single polynomial $\lambda$ such that for any field $K$ and any $\lambda_0 \in K^\times$, the set $\lambda^{-1}(\lambda_0) \subset W_N^{\on{top}}(K)$ is nonempty and consists of a single $(\on{SL}_n \times \on{SL}_{n+1})(K)$-orbit. Moreover, the stabilizer of any element of this orbit is trivial.
\end{prop}
\begin{proof}
  The proposition amounts to stating that the representation of $(\on{SL}_n \times \on{SL}_{n+1})(K)$ on $W_N^{\on{top}}(K)$ is ``prehomogeneous,'' and the proof is identical to that of~\cite[Proposition~3.1]{sqfrval2}. The main observation is that, for any integer $n_0 \geq 2$, the representation of $(\on{SL}_n \times \on{SL}_{n+1})(K)$ on $W_N^{\on{top}}(K)$ in the case $n = n_0$ is related to the representation of $(\on{SL}_n \times \on{SL}_{n+1})(K)$ on $W_N^{\on{top}}(K)$ in the case $n = n_0 - 1$ by what Sato and Kimura call a ``castling transform'' (see~\cite[\S2]{MR430336}). Since the property of being prehomogeneous is preserved under castling transforms, it suffices to prove the lemma in the case $n = 1$, where the action of $(\on{SL}_n \times \on{SL}_{n+1})(K)$ on $W_N^{\on{top}}(K)$ may be identified with the action of $\on{SL}_2(K)$ on $\on{Mat}_{2 \times 2}(K)$ by right-multiplication. The desired polynomial invariant function $\lambda$ for this action is given simply by the determinant.

  We now describe the invariant $Q$ for $n > 1$. Given a $\Z$-algebra $R$ and a pair $(A,B) \in W_N^{\on{top}}(R)$, we take $\lambda(A,B)$ to be the ``hyperdeterminant'' of $(A,B)$, which is defined explicitly as follows. For each $i \in \{1, \dots, n+1\}$, let $A^{(i)}$ and $B^{(i)}$ be the $n \times n$ matrices obtained from $A$ and $B$, respectively, by considering the top-right $n \times (n+1)$ block and deleting the $i^{\mathrm{th}}$ column. Then $\lambda(A,B)$ is the determinant of the $(n+1) \times (n+1)$ matrix whose row-$i$, column-$j$ entry is the $x^{n-j+1}y^{j-1}$-entry of the binary $n$-ic form $(-1)^{i+1}\det(xA^{(i)} - yB^{(i)})$. 
  
  That $Q^{-1}(Q_0)$ is nonempty for any $Q_0 \in K^\times$ follows by examining the following pair of matrices in $W_N^{\on{top}}(K)$:
  \begin{equation} \label{eq-qsech}
  \left(\left[\begin{array}{ccccc} & & & & 1 \\
  & & & 1 & \\ 
  & & \rddots & & \\
  & 1 & & &
  \end{array}\right],\left[\begin{array}{ccccc} & & & \pm Q_0 & \\
  & & 1 & & \\ 
  & \rddots  & & & \\
  1 & & & &
  \end{array}\right]\right)
  \end{equation}
  The pair in~\eqref{eq-qsech} evidently has $Q$-invariant $Q_0$ or $-Q_0$, depending on the choice of sign of the entry ``$\pm Q_0$.'' In fact, this pair gives an explicit section of the map $Q \colon W_N^{\on{top}} \to \mathbb{A}^1$, defined over $\Z$.  As for the claim about stabilizers, the action of $\on{SL}_2(K)$ by right-multiplication on matrices of nonzero determinant in $\on{Mat}_{2 \times 2}(K)$ obviously has trivial stabilizers, and stabilizers are preserved under castling transforms. 
\end{proof}

\subsubsection{Action of $L_N$ on $W_N^{\on{top},0}$}
\label{sec-lowering}

Let $W_N^{\on{top},0} \subset W_N^{\on{top}}$ be the linear subspace whose $R$-points \mbox{are given by}
$$W_N^{\on{top},0}(R) \defeq \big\{(A,B) \in W_N^{\on{top}}(R) : A_{ij} = 0\text{ if $i + j \leq 2n+1$ and }B_{ij} = 0\text{ if $i + j \leq 2n$}\big\}$$
for any $\Z$-algebra $R$. Let $L_N \subset \on{SL}_n \times \on{SL}_{n+1}$ be the lower-triangular subgroup. Then the action of $\on{SL}_n \times \on{SL}_{n+1}$ on $W_N^{\on{top}}$ restricts to an action of $L_N$ on $W_N^{\on{top},0}$. Given $(A,B) \in W_N^{\on{top},0}(R)$, we have the following formula for the $Q$-invariant up to sign:
\begin{equation} \label{eq-ogq} 
Q(A,B) = \pm\prod_{i = 1}^n A_{i,N+1-i}^{n+1-i} \prod_{i=1}^n B_{i,N-i}^i.
\end{equation}
We now describe this restricted action of $L_N(K)$ on $W_N^{\on{top},0}(K)$ for $K$ a field:

\begin{prop} \label{prop-qalltheway}
Let $K$ be a field, and let $Q_0 \in K^\times$. Then the set $Q^{-1}(Q_0) \cap W_N^{\on{top},0}(K)$ is nonempty and consists of a single $L_N(K)$-orbit. Moreover, the stabilizer of any element of this orbit is trivial.
\end{prop}
\begin{proof}
The nonemptiness statement follows immediately from the existence of the explicit section~\eqref{eq-qsech}, which has image contained in $W_N^{\on{top},0}$. The statement about stabilizers is obvious. 

As for the transitivity statement, take $(A,B) \in Q^{-1}(Q_0) \cap W_N^{\on{top},0}(K)$. It suffices to show that $(A,B)$ is $L_N(K)$-equivalent to~\eqref{eq-qsech}. First, by replacing $(A,B)$ with a suitable translate under the action of a diagonal element in $L_N(K)$, we may assume that $A_{ij} = 1$ for all $i + j = N+1$ and that $B_{ij} = 1$ for all $i + j = N$, except when $(i,j) = (1, N-1)$, in which case $B_{1(N-1)} = \pm Q_0$. Now, call a lower-triangular unipotent element $g \in \on{SL}_N(K)$ \emph{elementary} if $g_{ij} = 0$ for all but one pair $(i,j)$ with $i > j$. It is easy to verify by inspection that, by hitting $(A, B)$ with a suitable sequence of elementary unipotent elements of $L_N(K)$, we can successively clear out the values of the following matrix entries:
\begin{align} \label{eq-listcoffs1}
& B_{1(2n+1)},\, A_{2(2n+1)},\, B_{2(2n)},\, B_{2(2n+1)},\, \ldots,\, \\
&\qquad A_{k(2n+3-k)},\, \dots,\, A_{k(2n+1)},\, B_{k(2n+2-k)},\, \dots,\, B_{k(2n+1)},  \dots,\, \nonumber \\
&\qquad\qquad A_{n(n+3)},\, \dots,\, A_{n(2n+1)},\, B_{n(n+2)},\, \dots,\, B_{n(2n+1)}. \nonumber
\end{align}
We may thus assume that the matrix entries of $(A,B)$ that are listed in~\eqref{eq-listcoffs1} are all equal to $0$. 
But then $(A,B)$ is equal to~\eqref{eq-qsech}, which is sufficient.
\end{proof}

\begin{corollary} \label{cor-lng}
If two elements of $Q^{-1}(Q_0) \cap W_N^{\on{top},0}(K)$ are equivalent under the action of $g \in (\on{SL}_n \times \on{SL}_{n+1})(K)$, then we have $g \in L_N(K)$.
\end{corollary}
\begin{proof}
Suppose for some $g_1 \in (\on{SL}_n \times \on{SL}_{n+1})(K)$ and elements $(A_1, B_1), (A_2, B_2) \in Q^{-1}(Q_0) \cap W_N^{\on{top},0}(K)$ we have $g_1 \cdot (A_1, B_1) = (A_2, B_2)$. By Proposition~\ref{prop-qalltheway}, there exists $g_2 \in L_N(K)$ such that $g_2 \cdot (A_1, B_1) = (A_2, B_2)$, so $g_2^{-1}g_1$ stabilizes $(A_1, B_1)$, but by Proposition~\ref{prop-castle}, the stabilizer of $(A_1, B_1)$ in $(\on{SL}_n \times \on{SL}_{n+1})(K)$ is trivial, so $g_1 = g_2$, as necessary.
\end{proof}

\subsubsection{Fundamental domain for $(\on{SL}_n \times \on{SL}_{n+1})(\Z_p) \curvearrowright W_N^{\on{top}}(\Z_p)$} \label{sec-fundzits}

Using the results of \S\S\ref{sec-221}--\ref{sec-lowering}, we now construct a fundamental domain for the action of $(\on{SL}_n \times \on{SL}_{n+1})(\Z_p)$ on $W_N^{\on{top}}(\Z_p)$; this fundamental domain plays a crucial role in the proof of Theorem~\ref{thm-main1} (see \S\ref{sec-pfthmn1}).\footnote{Note that by a ``fundamental domain'' for the action of a group on a set, we mean a subset that contains exactly one element of each orbit.}

We start by choosing a convenient partition of $W_N^{\on{top}}(\Z_p)$ into subsets indexed by pairs of nonnegative-integer-vectors of length $n$. Given $\vec{a} = (a_1, \dots, a_n),\,\vec{b} = (b_1, \dots, b_n) \in \mathbb{N}^n$, define a subset $\mc{L}_{\vec{a},\vec{b}}(p)$ as follows:
\begin{align*}
\mc{L}_{\vec{a},\vec{b}}(p) & \defeq (\on{SL}_n \times \on{SL}_{n+1})(\Z_p) \times \\
& \qquad\qquad \left\{(A,B) \in W_N^{\on{top},0}(\Z_p) : \begin{array}{c} \nu_p(A_{ij}) = a_{i} \text{ for $(i,j)$ with $i + j = N+1$,} \\ \nu_p(B_{ij}) = b_i \text{ for $(i,j)$ with $i + j = N$}\end{array} 
\right\}
\end{align*}
Take $(A,B) \in \mc{L}_{\vec{a},\vec{b}}(p)$. By the proof of Proposition~\ref{prop-qalltheway}, we can use the action of $(\on{SL}_n \times \on{SL}_{n+1})(\Q_p)$ to make $A_{ij} = 0$ for each $(i,j)$ with $i + j > N+1$ and $B_{ij} = 0$ for each $(i,j)$ with $i+j > N$, without changing $A_{ij}$ for each $(i,j)$ with $i + j = N+1$ and $B_{ij}$ for each $(i,j)$ with $i+j = N$. Then one readily computes using~\eqref{eq-ogq} that
\begin{equation} \label{eq-formforq}
|Q(A,B)|_p^{-1} = \prod_{i = 1}^n p^{(n+1-i)a_i + ib_i}.
\end{equation}
The following result gives the desired fundamental domain in terms of the sets $\mc{L}_{\vec{a},\vec{b}}(p)$:
\begin{prop} \label{prop-padicfund}
For each $\vec{a},\vec{b} \in \mathbb{N}^n$, a fundamental domain for the action of $(\on{SL}_n \times \on{SL}_{n+1})(\Z_p)$ on the set $\mc{L}_{\vec{a},\vec{b}}(p)$ is given by
\small
\begin{equation} \label{eq-funder}
\left\{(A,B) \in W_N^{\on{top},0}(\Z_p) : \begin{array}{c} A_{ij} = p^{a_i} \text{ for $(i,j)$ with $i+j = N+1$,} \\ A_{ij} \in \{0,\dots,p^{ a_{N+1-j}}-1\} \text{ for $(i,j)$ with $i + j > N+1$ and $i \leq n$,} \\B_{ij} = p^{b_i} \text{ for $(i,j)$ with $i+j = N$ and $i > 1$,} \\ \nu_p(B_{1(N-1)}) = b_1, \\ B_{ij} \in \{0, \dots, p^{b_i}-1\} \text{ for $(i,j)$ with $i + j > N$ and $i \leq n$}\end{array}\right\}
\end{equation}
\normalsize
A fundamental domain for the action of $(\on{SL}_n \times \on{SL}_{n+1})(\Z_p)$ on the set $\{(A,B) \in W_N^{\on{top}}(\Z_p) : Q(A,B) \neq 0\}$ is given by the $($disjoint$)$ union of the set~\eqref{eq-funder} over all $\vec{a},\vec{b} \in \mathbb{N}^n$.
\end{prop}
\begin{proof}
First note that we have
\begin{equation} \label{eq-rhs}
\mc{L}_{\vec{a}, \vec{b}}(p) \cap W_N^{\on{top},0}(\Z_p) = \left\{(A,B) \in W_N^{\on{top},0}(\Z_p) : \begin{array}{c} \nu_p(A_{ij}) = a_{i} \text{ for $(i,j)$ with $i + j = N+1$,} \\\nu_p(B_{ij}) = b_i \text{ for $(i,j)$ with $i + j = N$}\end{array} 
\right\}.
\end{equation}
Indeed, the left-hand side of~\eqref{eq-rhs} obviously contains the right-hand side. As for the reverse containment, if $(A,B) \in \mc{L}_{\vec{a}, \vec{b}}(p) \cap W_N^{\on{top},0}(\Z_p)$, then there exist an element $g \in (\on{SL}_n \times \on{SL}_{n+1})(\Z_p)$ and a pair $(A',B')$ belonging to the right-hand side of~\eqref{eq-rhs} such that $(A,B) = g \cdot (A',B')$. By Corollary~\ref{cor-lng}, we must have $g \in L_N(\Z_p)$, so since the action of $L_N(\Z_p)$ preserves the right-hand side of~\eqref{eq-rhs}, the pair $(A,B)$ must belong to it as well. 

 The next step is to show that the action of $(\on{SL}_n \times \on{SL}_{n+1})(\Z_p)$ can be used to move points in $W_N^{\on{top}}(\Z_p)$ into the linear subspace $W_N^{\on{top},0}(\Z_p)$:

\begin{lemma} \label{lem-meettri}
Let $Q_0 \in \Z_p \smallsetminus \{0\}$. Then every $(\on{SL}_n \times \on{SL}_{n+1})(\Z_p)$-orbit on $Q^{-1}(Q_0) \subset W_N^{\on{top}}(\Z_p)$ meets $W_N^{\on{top},0}(\Z_p)$.
\end{lemma}
\begin{proof}[Proof of Lemma~\ref{lem-meettri}]
The proof is similar to that of~\cite[Lemma~206]{swathesis}. Take $(A,B) \in Q^{-1}(Q_0) \cap W_N^{\on{top}}(\Z_p)$. Then there exists $g \in (\on{SL}_n \times \on{SL}_{n+1})(\Q_p)$ such that $g \cdot (A,B) \in W_N^{\on{top},0}(\Z_p)$, because the section~\eqref{eq-qsech} of $Q$ is defined over $\Z_p$ with image contained in $W_N^{\on{top},0}$, and because the elements of the set $Q^{-1}(Q_0) \cap W_N^{\on{top}}(\Z_p)$ belong to the same $(\on{SL}_n \times \on{SL}_{n+1})(\Q_p)$-orbit by Proposition~\ref{prop-castle}. 

By the $p$-adic Iwasawa decomposition, we have that 
$$(\on{SL}_n \times \on{SL}_{n+1})(\Q_p) = L_N(\Q_p)(\on{SL}_n \times \on{SL}_{n+1})(\Z_p),$$
so there exists $g_1 \in L_N(\Q_p)$ and $g_2 \in (\on{SL}_n \times \on{SL}_{n+1})(\Z_p)$ such that $g = g_1g_2$. But $L_N$ acts on $W_N^{\on{top},0}$ (see \S\ref{sec-lowering} for further details), so $g_2 \cdot (A,B) \in W_N^{\on{top},0}(\Q_p)$. Since $g_2 \in (\on{SL}_n \times \on{SL}_{n+1})(\Z_p)$ and $(A,B) \in W_N^{\on{top}}(\Z_p)$, we must in fact have that $g_2 \cdot (A,B) \in W_N^{\on{top},0}(\Z_p)$.
\end{proof}
Now, by Lemma~\ref{lem-meettri} and Corollary~\ref{cor-lng}, it suffices to show that~\eqref{eq-funder} is a fundamental domain for the action of $L_N(\Z_p)$ on $\mc{L}_{\vec{a},\vec{b}}(p) \cap W_N^{\on{top},0}(\Z_p)$, but this follows by adapting the proof of Proposition~\ref{prop-qalltheway} to work over $\Z_p$. Take $(A,B) \in \mc{L}_{\vec{a},\vec{b}}(p)\cap W_N^{\on{top},0}(\Z_p)$. First, instead of using diagonal transformations to make $A_{ij} = 1$ for $(i,j)$ with $i+j = N+1$ and $B_{ij} = 1$ for $(i,j)$ with $i+j = N$, except when $(i,j) = (1,N-1)$, we use these transformations to make $A_{ij} = p^{a_i}$ for $(i,j)$ with $i+j = N+1$ and $B_{ij} = p^{b_i}$ for $(i,j)$ with $i+j = N$, except when $(i,j) = (1,N-1)$. Then, instead of using elementary unipotent transformations to make the remaining matrix entries zero, we use these transformations to reduce these matrix entries modulo $p^{a_1}, \dots, p^{a_n}, p^{b_1}, \dots, p^{b_n}$. This establishes that each orbit for the action of $(\on{SL}_n \times \on{SL}_{n+1})(\Z_p)$ on $\mc{L}_{\vec{a},\vec{b}}(p)$ meets \mbox{the set~\eqref{eq-funder} at least once. }

On the other hand, if $(A,B), (A',B')$ belong to the set~\eqref{eq-funder} and there exists an element $g \in (\on{SL}_n \times \on{SL}_{n+1})(\Z_p)$ such that $(A,B) = g \cdot (A',B')$, then by Corollary~\ref{cor-lng}, we must have $g \in L_N(\Z_p)$. Write $g = g_1g_2$, where $g_1$ is unipotent and $g_2$ is diagonal. Denote the diagonal entries of $g_2$ by $(g_2)_{ii}$. Observe that we have the equalities 
\begin{equation} \label{eq-aacomp}
(g_2)_{ii}(g_2)_{jj}A_{ij}' = (g\cdot A')_{ij} = A_{ij} = p^{a_i} = A_{ij}'
\end{equation}
for $(i,j)$ with $i+j = N+1$ and $i \leq n$, and that we have
\begin{equation} \label{eq-bbcomp}
(g_2)_{ii}(g_2)_{jj}B_{ij}' = (g \cdot B')_{ij} =  B_{ij} = p^{b_i} = B_{ij}'
\end{equation}
for $(i,j)$ with $i + j = N$ and $1 < i \leq n$. Combining~\eqref{eq-aacomp} and~\eqref{eq-bbcomp} with the fact that $Q(A,B) = Q(A',B')$ along with the formula~\eqref{eq-ogq}, we see that~\eqref{eq-bbcomp} holds when $i = 1$ too. Then combining~\eqref{eq-aacomp} and~\eqref{eq-bbcomp} with the condition $\det g_2 = 1$ yields the following system of equations:
\begin{align*}
    (g_2)_{ii} (g_2)_{(N-i)(N-i)} = (g_2)_{ii} (g_2)_{(N+1-i)(N+1-i)} = \prod_{j = 1}^N (g_2)_{jj} = 1 \quad \text{for each $i \in \{1, \dots, n\}$}.
\end{align*}
By comparing the products $\prod_{i = 1}^n(g_2)_{ii} (g_2)_{(N-i)(N-i)}$ and $\prod_{j = 1}^N (g_2)_{jj}$, we deduce that $(g_2)_{NN} = 1$, from which it follows that all $(g_2)_{ii} = 1$, and hence that $g_2 = \on{id}$.

Next, denote the entries of $g_1$ by $(g_1)_{ij}$. Then the condition $B = g_1 \cdot B'$ implies that $B_{1(2n+1)} \equiv B_{1(2n+1)}' \pmod{p^{b_1}}$, so the condition $B_{1(2n+1)},B_{1(2n+1)}' \in \{0, \dots, p^{b_1}-1\}$ forces $B_{1(2n+1)} = B_{1(2n+1)}'$ and hence that $(g_1)_{N(N-1)} = 0$. Since $(g_1)_{N(N-1)} = 0$, the condition $A = g_1 \cdot A'$ implies that $A_{2(2n+1)} \equiv A_{2(2n+1)}' \pmod{p^{a_{1}}}$, so the condition $A_{2(2n+1)},A_{2(2n+1)}' \in \{0, \dots, p^{a_{1}}-1\}$ forces $A_{2(2n+1)} = A_{2(2n+1)}'$ and hence that $(g_1)_{21} = 0$. Continuing in this manner according to the sequence of matrix entries in~\eqref{eq-listcoffs1}, we see that the matrix entries of $(A,B)$ and $(A',B')$ coincide.
\end{proof}

\subsubsection{The subgroup $G_N \subset \on{SL}_N$} \label{sec-thesub}

Let $G_N \subset \on{SL}_N$ be the subgroup whose $R$-points are given by 
$$G_N(R) \defeq \big\{g \in \on{SL}_{N}(R) : g_{ij} = 0 \text{ for all $(i,j)$ such that }i \leq n \text{ and } j\geq n+1\big\}$$ 
for any $\Z$-algebra $R$.

An algebraic group $G$ is said to have \emph{class number $1$ over $\Q$} if the group $G(\mathbb{A}_\Q)$ of adelic rational points is the ``Frobenius'' product of the subgroup $G(\Q)$ of rational points with the subgroup $G(\mathbb{A}_{\Z})$ of adelic integral points --- i.e., we have $G(\mathbb{A}_{\Q}) = G(\Q)G(\mathbb{A}_{\Z})$. The following result establishes that this property holds for $G = G_N$:

\begin{prop} \label{prop-classone}
  The algebraic group $G_N$ has class number $1$ over $\Q$.
\end{prop}
\begin{proof}
 The idea of the proof is to realize $G_N$ as the Frobenius product of two subgroups, each of which has class number $1$ over $\Q$, and then to use this product structure to deduce that $G_N$ itself has class number $1$ over $\Q$.

 Let $H_1$ be the lower-triangular unipotent subgroup of $G_N$ whose $R$-points are given by 
 $$H_1(R) \defeq \left\{h \in G_N(R) : \begin{array}{c}h_{ii} = 1 \text{ for all $i$ and } \\ h_{jk} = 0 \text{ if $j < k$ or $j > k \geq n+1$ or $n \geq j > k$}\end{array}\right\}$$
 for any $\Z$-algebra $R$, and let $H_2$ be the block-diagonal subgroup of $G_N$ whose \mbox{$R$-points are given by}
 $$H_2(R) \defeq \{g \in \on{GL}_n(R) \times \on{GL}_{n+1}(R) : \det g = 1 \} \subset G_N(R).$$
  \begin{lemma} \label{lem-prodhgl2}
   The algebraic group $G_N$ is the Frobenius product of its subgroups $H_1$ and $H_2$; i.e., for any $\Z$-algebra $R$ we have $G_N(R) = H_1(R)H_2(R)$. Moreover, we have that $H_1(R) \cap H_2(R) = 1$.
  \end{lemma}
  \begin{proof}[Proof of Lemma~\ref{lem-prodhgl2}]
 The second claim is clear from the definitions of $H_1$ and $H_2$. As for the first claim, take $g \in G_N(R)$, and write it in box form as follows:
 \begin{equation*} 
 g = \left[\begin{array}{c|c} g' & 0 \\ \hline g''' & g'' \end{array}\right],
 \end{equation*}
 where $g' \in \on{GL}_n(R)$, $g'' \in \on{GL}_{n+1}(R)$, and $g''' \in \on{Mat}_{(n+1) \times n}(R)$. Then it is easy to see that
 \begin{equation*}
 g = \left[\begin{array}{c|c} \on{id} & 0 \\ \hline g''' {g'}^{-1} & \on{id} \end{array}\right] \times \left[\begin{array}{c|c} g' & 0 \\ \hline 0 & g'' \end{array}\right],
 \end{equation*}
 where by ``$\on{id}$'' (resp., ``$0$'') we mean the identity (resp., zero) matrix of the relevant dimensions.
  \end{proof}
  \noindent The group $H_1$ has class number $1$ over $\Q$ because it is isomorphic to the additive group $\mathbb{G}_a^{n^2+n}$; that $H_2$ has class number $1$ over $\Q$ follows immediately from the fact that the same holds for the groups $\on{GL}_n$ and $\on{GL}_{n+1}$. The next lemma gives a criterion under which the Frobenius product of two groups of class number $1$ over $\Q$ is itself a group of class number $1$ over $\Q$:
\begin{lemma} \label{lem-genclass}
Let $\Gamma$ be an algebraic group over $\Z$ that is the Frobenius product of two sub-algebraic-groups $\Gamma_1$ and $\Gamma_2$, both of which have class number $1$ over $\Q$, and suppose that $\Gamma_1(\mathbb{A}_{\Z})\Gamma_2(\Q) \subset \Gamma_2(\Q)\Gamma_1(\mathbb{A}_\Q)$. Then $\Gamma$ has class number $1$ over $\Q$.
\end{lemma}
\begin{proof}[Proof of Lemma~\ref{lem-genclass}]
Clearly, $\Gamma(\mathbb{A}_\Q) \supset \Gamma(\Q)\Gamma(\mathbb{A}_\Z)$. As for the reverse inclusion, we have that
\begin{align*}
\Gamma(\mathbb{A}_\Q) = \Gamma_1(\mathbb{A}_\Q)\Gamma_2(\mathbb{A}_{\Q}) & = \Gamma_1(\Q)\Gamma_1(\mathbb{A}_\Z)\Gamma_2(\Q)\Gamma_2(\mathbb{A}_\Z) \\
& \subset \Gamma_1(\Q)\Gamma_2(\Q)\Gamma_1(\mathbb{A}_\Q)\Gamma_2(\mathbb{A}_{\Z}) \\
& = \Gamma_1(\Q)\Gamma_2(\Q)\Gamma_1(\Q)\Gamma_1(\mathbb{A}_\Z)\Gamma_2(\mathbb{A}_{\Z}) = \Gamma(\Q) \Gamma(\mathbb{A}_\Z). \qedhere
\end{align*}
\end{proof}
\noindent By Lemma~\ref{lem-genclass}, it now suffices to check that the criterion  $H_1(\mathbb{A}_{\Z})H_2(\Q) \subset H_2(\Q)H_1(\mathbb{A}_\Q)$ holds. This is an immediate consequence of the following matrix identity:
\begin{equation} \label{eq-matident}
\left[\begin{array}{c|c} \on{id} & 0 \\ \hline g''' & \on{id} \end{array}\right] \times \left[\begin{array}{c|c} g' & 0 \\ \hline 0 & g'' \end{array}\right] = \left[\begin{array}{c|c} g' & 0 \\ \hline 0 & g'' \end{array}\right] \times \left[\begin{array}{c|c} \on{id} & 0 \\ \hline {g''}^{-1}g''' g'  & \on{id} \end{array}\right]. \qedhere
\end{equation}
\end{proof}

\subsubsection{Action of $G_N$ on $W_N^0$} \label{sec-thesubway}

 Let $W_N^0 \subset W_N$ be the linear subspace whose $R$-points are defined by
$$W_N^0(R) \defeq \big\{(A,B) \in W_N(R) : A_{ij} = B_{ij} = 0 \text{ if $i,j \leq n$}\big\}$$
for any $\Z$-algebra $R$, and notice that, when $R$ is an integral domain, every pair $(A,B) \in W_N^0(R)$ is reducible. The action of $\on{SL}_{N}$ on $W_N$ restricts to an action of $G_N$ on $W_N^0$.

We extend the definition of the $\lambda$-invariant to any pair $(A,B) \in W_N^0(R)$ by defining $\lambda(A,B)$ to be the $\lambda$-invariant of the projection of $(A,B)$ onto the linear subspace $W_N^{\on{top}}(R)$ that sends the $(n+1) \times (n+1)$ entries in the bottom-right of $A$ and $B$ to zero. A fundamental property of the $Q$-invariant is that it divides the discriminant of the invariant binary form to order two --- i.e., we have $Q(A,B)^2 \mid \on{disc}(\on{inv}(A,B))$ for any $(A,B) \in W_N^0(R)$; for a proof, see~\cite[Theorem 3.5]{sqfrval2}.

The function $\lambda$ is notably \emph{not} $G_N$-invariant, despite being invariant under the action of the subgroup $\on{SL}_n \times \on{SL}_{n+1} \subset G_N$; see, e.g., the proof of Proposition~\ref{prop-jac} (to follow). We note that the explicit section $\sigma_0$ described in \S\ref{sec-repert} has image contained in the locus of points in $W_N^0$ with $\lambda$-invariant equal to $(-1)^{n+1}$.

We now describe the action of $G_N$ on $W_N^0$ over a field $K$:
\begin{prop} \label{prop-ver2}
Let $K$ be a field, and let $f \in U_N(K)$. Then the set $\{(A,B) \in \on{inv}^{-1}(f) \cap W_N^0(K) : Q(A,B) \neq 0\}$ is nonempty and consists of a single $G_N(K)$-orbit. Moreover, the stabilizer of any element of this orbit is trivial.
\end{prop}
\begin{proof}
That $\on{inv}^{-1}(f) \cap W_N^0(R)$ is nonempty holds over any $\Z$-algebra $R$ 
because of the existence of the aforementioned explicit section $\sigma_0$.

To see that $\{(A,B) \in \on{inv}^{-1}(f) \cap W_N^0(K) : Q(A,B) \neq 0\}$ consists of a single $G_N(K)$-orbit, take $(A_0,B_0) \in \on{inv}^{-1}(f) \cap W_N^0(K)$ with $Q_0 = Q(A_0, B_0) \in K^\times$. We claim that there exists $g_1 \in (\on{SL}_n \times \on{SL}_{n+1})(K) \subset G_N(K)$ such that $(A, B) \defeq g_1 \cdot (A_0,B_0)$ has the following shape, where empty entries are used to denote zeros, and star entries are used to denote numbers whose particular values are irrelevant:
\begin{equation} \label{eq-a1b1}
(A,B) = \scriptstyle\left(\left[\begin{array}{cccc|ccccc}
& & & & & & & & 1 \\
& & & & & & & 1 & \\
& & & & & & \rddots & & \\
& & & & & 1 &  & & \\ \hline
& & & & * & * & \cdots & * & *  \\ 
& & & 1 & * & * & \cdots & * & *  \\
& & \rddots & & \vdots & \vdots & \ddots & \vdots & \vdots  \\
& 1 & & & * & * & \cdots & * & *  \\
1 & & & & * & * & \cdots & * & * \end{array}\right], \left[\begin{array}{cccc|ccccc}
& & & & & & & \pm\lambda_0 & \\
& & & & & & 1 & & \\
& & & & & \rddots & & & \\
& & & & 1 &  & & & \\ \hline
& & & 1 & * & \cdots & * & * & * \\
& & \rddots & & \vdots & \ddots & \vdots & \vdots & \vdots \\
& 1 & & & * & \cdots & * & * & * \\
\pm\lambda_0 & & & & * & \cdots & * & * & * \\ & & & & * & \cdots & * & * & * \end{array}\right]\right)
\end{equation}
In the above, we have inserted horizontal and vertical lines immediately after row and column $n$ in the matrices $A$ and $B$. The claim follows from Proposition~\ref{prop-castle} upon observing that by choosing the sign of the entry ``$\pm Q_0$'' appropriately, we can arrange for $Q(A,B) = Q_0$. In fact, by postcomposing $g_1$ with a suitable diagonal element in $G_N(K)$, we may take $(A, B)$ to be of the form~\eqref{eq-a1b1}, with the entries ``$\pm Q_0$'' replaced by ``$1$.''

Now, just as in the proof of Proposition~\ref{prop-qalltheway}, by hitting $(A, B)$ with a suitable sequence of elementary unipotent elements of $G_N(K)$, we can successively clear out the values of the following matrix entries:
\begin{align} \label{eq-listcoffs2}
& A_{(n+1)(n+2)},\, \dots,\, A_{(n+1)(2n+1)},\, B_{(n+1)(n+2)}, \, \ldots,\, B_{(n+1)(2n+1)},\, A_{(n+2)(n+3)},\, \ldots,\, A_{(n+2)(2n+1)},\, \ldots, \, \\
& \qquad B_{k(k+1)},\, \ldots,\,B_{k(2n+1)},\, A_{(k+1)(k+2)},\, \ldots,\, A_{(k+1)(2n+1)},\, \ldots,\, \nonumber \\
& \qquad\qquad B_{(2n-1)2n},\, B_{(2n-1)(2n+1)},\, A_{(2n)(2n+1)},\, B_{(2n)(2n+1)}. \nonumber
\end{align}
We may thus assume that the matrix entries of $(A,B)$ that are listed in~\eqref{eq-listcoffs2} are all equal to $0$. But then $(A,B)$ lies in the image of the section $\sigma_0$, which is sufficient.

As for the claim about stabilizers, let $(A_0, B_0)$ be as above, and suppose $g \in G_N(K)$ stabilizes $(A_0, B_0)$. By Lemma~\ref{lem-prodhgl2}, we may write $g = g_1g_2$, where $g_i \in H_i(K)$ for each $i \in \{1,2\}$. Since the projection map $W_N^0 \to W_N^{\on{top}}$ is invariant under the action of $H_1(K)$, it follows that $g_2$ must stabilize the projection of $(A_0, B_0)$ onto $W_N^{\on{top}}(K)$. Thus, by Proposition~\ref{prop-castle}, we have $g_2 = \on{id}$. Now, it is clear by inspection that the action of $H_1(K)$ on $W_N^0(K)$ has trivial stabilizers, so $g_1 = \on{id}$, and thus $g = g_1g_2 = \on{id}$ too.
%
\end{proof}

We now work over $\Z_p$ for a prime $p$. Note that the $p$-adic absolute value $|Q(A,B)|_p$ is invariant under the action of $G_N(\Z_p)$. The following result describes the action of $G_N(\Z_p)$ on the locus of pairs in $W_N^0(\Z_p)$ with $p$-adic unit $Q$-invariant. 

\begin{prop} \label{prop-singleQ}
Let $f \in U_N(\Z_p)$. Then the set $\{(A,B) \in \on{inv}^{-1}(f) \cap W_N^0(\Z_p) : |Q(A,B)|_p = 1\}$ is nonempty and consists of a single $G_N(\Z_p)$-orbit. Moreover, the stabilizer of any element of this orbit is trivial.
\end{prop}
\begin{proof}
The nonemptiness statement follows immediately from the existence of the explicit section $\sigma_0$, which has image contained in the locus of points in $W_N^0$ with $Q$-invariant equal to $\pm 1$.

 Next, let $W_N^{00} \subset W_N^0$ be the linear subspace whose $R$-points are given by
$$W_N^{00}(R) \defeq \big\{(A,B) \in W_N^0(R) : A_{ij} = 0\text{ if $i + j \leq 2n+1$ and }B_{ij} = 0\text{ if $i + j \leq 2n$}\big\}.$$
 We remark that the explicit section $\sigma_0$ described in \S\ref{sec-repert} has image contained in $W_N^{00}$. For the transitivity statement, by Lemma~\ref{lem-meettri} it suffices to show that $\{(A,B) \in \on{inv}^{-1}(f) \cap W_N^{00}(\Z_p) : |\lambda(A,B)|_p = 1\}$ is contained in a single $G_N(\Z_p)$-orbit, and the proof of Proposition~\ref{prop-ver2} can be easily adapted to verify this.

Finally, the statement about stabilizers follows immediately from Proposition~\ref{prop-ver2}, which implies that the stabilizer in $G_N(\Q_p)$, and hence in $G_N(\Z_p)$, of any $(A,B) \in \on{inv}^{-1}(f) \cap W_N^0(\Z_p)$ is trivial. This concludes the proof of Proposition~\ref{prop-singleQ}.
\end{proof}

\subsubsection{Fundamental domain for $G_N(\Z_p) \curvearrowright W_N^0(\Z_p)$}

Using the results of the preceding subsubsections, we now construct a fundamental domain for the action of $G_N(\Z_p)$ on $W_N^0(\Z_p)$. Given $\vec{a} = (a_1, \dots, a_n), \, \vec{b} = (b_1, \dots, b_n) \in \mathbb{N}^n$, define a subset $\mc{W}_{\vec{a},\vec{b}}(p)$ as follows:
\begin{align} \label{eq-defofws}
\mc{W}_{\vec{a},\vec{b}}(p) & \defeq \left\{(A,B) \in W_N^{00}(\Z_p) : \begin{array}{c} A_{ij} = p^{a_{i}} \text{ for $(i,j)$ with $i + j = N+1$,} \\ B_{ij} = p^{b_i} \text{ for $(i,j)$ with $i + j = N$}\end{array} 
\right\}
\end{align}
The following result gives the desired fundamental domain in terms of the sets $\mc{W}_{\vec{a},\vec{b}}(p)$:
\begin{prop} \label{prop-padicfund2}
A fundamental domain for the action of $G_N(\Z_p)$ on the set $\{(A,B) \in W_N^{0}(\Z_p) : Q(A,B) \neq 0\}$ is given by $$\bigsqcup_{\substack{\vec{a},\vec{b}\in \mathbb{N}^n}} \mc{F}_{\vec{a},\vec{b}}(p),$$ where the sets $\mc{F}_{\vec{a},\vec{b}}(p)$ are defined as follows:
\small
\begin{equation} \label{eq-funder3}
\mc{F}_{\vec{a},\vec{b}}(p) \defeq \left\{(A,B) \in \mc{W}_{\vec{a}, \vec{b}}(p) : \hspace*{-3pt}\begin{array}{c} A_{ij} \in \{0,\dots,p^{ a_{N+1-j}}-1\} \text{ for $(i,j)$ with $i + j > N+1$, $i < j$,} \\ B_{ij} \in \{0, \dots, p^{b_i}-1\} \text{ for $(i,j)$ with $i + j > N$, $i \leq n$,} \\ B_{ij} \in \{0, \dots, p^{b_{N-i}}-1\} \text{ for $(i,j)$ with $i + j > N$, $n < i < j$,}\end{array}\right\}
\end{equation}
\normalsize
\end{prop}
\begin{proof}
This follows by adapting the proof of Proposition~\ref{prop-ver2} to work over $\Z_p$ (just as we adapted the proof of Proposition~\ref{prop-qalltheway} to obtain Proposition~\ref{prop-padicfund}). It follows from Proposition~\ref{prop-padicfund} that every $(A',B') \in W_N^0(\Z_p)$ with $Q(A',B') \neq 0$ is $G_N(\Z_p)$-equivalent to an element $(A,B) \in W_N^{00}(\Z_p)$ with $A_{ij} = p^{a_i}$ for $(i,j)$ with $i + j = N+1$, with $B_{ij} = p^{b_i}$ for $(i,j)$ with $i + j = N$ and $i > 1$, and $B_{1(N-1)} = u \cdot p^{b_1}$ for some $u \in \Z_p^\times$, where $\vec{a},\, \vec{b} \in \mathbb{N}^n$. Using the action of the diagonal matrix with diagonal entries $(u^{-1}, 1, \dots, 1, u) \in G_N(\Z_p)$, we can further arrange that $(A,B) \in \mc{W}_{\vec{a},\vec{b}}(p)$. So take $(A,B) \in \mc{W}_{\vec{a},\vec{b}}(p)$. 
Instead of using elementary unipotent transformations to make the remaining nondiagonal matrix entries zero, we use these transformations to reduce these matrix entries modulo $p^{a_1}, \dots, p^{a_n}, p^{b_1}, \dots, p^{b_n}$. This establishes that each orbit for the action of $G_N(\Z_p)$ on $\mc{W}_{\vec{a},\vec{b}}(p)$ meets $\mc{F}_{\vec{a},\vec{b}}(p)$ at least once. 

On the other hand, suppose we have $(A,B), (A',B') \in \mc{F}_{\vec{a},\vec{b}}(p)$ and $g \in G_N(\Z_p)$ such that $(A,B) = g \cdot (A',B')$. Write $g = g_1g_2$, where $g_i \in H_i(\Z_p)$ for each $i \in \{1,2\}$, and factor $g_2$ as $g_2 = g_2'g_2''$, where $g_2'$ is a diagonal matrix with diagonal entries $(u,1, \dots, 1, u^{-1})$ for some $u \in \Z_p^\times$, and where $g_2'' \in (\on{SL}_n \times \on{SL}_{n+1})(\Z_p)$. Since $g_1^{-1} \cdot(A, B)$ and $(A',B')$ both lie in $\mc{W}_{\vec{a},\vec{b}}(p)$, we have by~\eqref{eq-ogq} along the invariance of $Q$ under the action of $\on{SL}_n \times \on{SL}_{n+1}$ that $$Q(g_1^{-1} \cdot(A, B)) = Q(g_2'' \cdot (A',B')).$$ But $g_1^{-1} \cdot (A,B) = g_2' \cdot (g_2'' \cdot (A',B'))$, so $$u \cdot Q(g_2'' \cdot (A',B')) = Q(g_2' \cdot (g_2'' \cdot (A',B')) = Q(g_2'' \cdot (A',B')),$$ from which it follows that $u = 1$ and $g_2 = g_2'' \in (\on{SL}_n \times \on{SL}_{n+1})(\Z_p)$. Then, since the projections of $g_1^{-1}\cdot (A,B)$ and $(A',B')$ onto $W_N^{\on{top},0}(\Z_p)$ belong to the fundamental domain~\eqref{eq-funder}, it follows that $g_2 = \on{id}$.

Next, denote the entries of $g_1$ by $(g_1)_{ij}$. The condition $A = g_1 \cdot A'$ implies that $A_{(n+1)(n+2)} \equiv A_{(n+1)(n+2)}' \pmod{p^{a_n}}$, and then the condition $A_{(n+1)(n+2)}, A_{(n+1)(n+2)}' \in \{0,\dots, p^{a_n}-1\}$ forces $A_{(n+1)(n+2)} = A_{(n+1)(n+2)}'$ and hence that $(g_1)_{(n+1)n} = 0$. As $(g_1)_{(n+1)n} = 0$, the condition $A = g_1 \cdot A'$ implies that $A_{(n+1)(n+3)} \equiv A_{(n+1)(n+3)}' \pmod{p^{a_{n-1}}}$, so the condition  $A_{(n+1)(n+3)}, A_{(n+1)(n+3)}' \in \{0,\dots, p^{a_{n-1}}-1\}$ forces $A_{(n+1)(n+3)} = A_{(n+1)(n+3)}'$ and hence that $(g_1)_{(n+1)(n-1)} = 0$. Continuing in this manner according to the sequence of matrix entries in~\eqref{eq-listcoffs2} (and using the condition $B = g_1 \cdot B'$ when appropriate), we deduce that $g_1 = \on{id}$, and hence that $(A,B) = (A',B')$.
\end{proof}

We now cover each of the sets $\mc{F}_{\vec{a},\vec{b}}(p)$ with a finite disjoint union of images of sections of the map $\on{inv}$. Fix $\vec{a},\vec{b} \in \mathbb{N}^n$. Fix the matrix entries $A_{ij} = p^{a_i}$ for $(i,j)$ with $(i,j) = N+1$, $B_{ij} = p^{b_i}$ for $(i,j)$ with $i+j = N$, just as in~\eqref{eq-defofws}. Further fix matrix entries $A_{ij} \in \{0,\dots, p^{a_{N+1-j}}-1\}$ for each $(i,j)$ with $i + j > N+1$ and $i < j$, $B_{ij} \in \{0, \dots, p^{b_i}-1\}$ for each $(i,j)$ with $i + j > N$ and $i \leq n$, and $B_{ij} \in \{0, \dots, p^{b_{N-i}}-1\}$ for each $(i,j)$ with $i + j > N$ and $n < i < j$, just as in~\eqref{eq-funder3}. Call the data of these chosen matrix entries $(A^\circ, B^\circ)$; then we may regard $(A^\circ, B^\circ)$ as a function on $\Z_p^{N+1}$ which takes a tuple $t = \big(A_{(n+1)(n+1)}, \dots, A_{NN}, B_{(n+1)(n+1)}, \dots, B_{NN}\big) \in \Z_p^{N+1}$ and produces the pair of matrices $(A,B)$ with the previously fixed off-diagonal matrix entries, and with diagonal entries given by the components of $t$. 

We may now speak of the quantity $\on{inv}(A^\circ, B^\circ)$, which is a binary form whose coefficients may be thought of as an affine linear transformation of the $N+1$ indeterminates $$A_{(n+1)(n+1)}, \dots, A_{NN}, B_{(n+1)(n+1)}, \dots, B_{NN}.$$ 
Write the coefficients of the indeterminates in an $(N+1) \times (N+1)$ matrix $M(A^\circ, B^\circ)$ as follows: the $i^{\mathrm{th}}$ row corresponds to the coefficient of $x^{N-i+1}y^{i-1}$ in $\on{inv}(A^\circ, B^\circ)$, and the columns correspond to the coefficients of $A_{(n+1)(n+1)}, B_{(n+1)(n+1)}, A_{(n+2)(n+2)}, B_{(n+2)(n+2)},\dots, A_{NN}, B_{NN}$, in that order. Then one verifies by inspection that $M(A^\circ, B^\circ)$ is lower-triangular, and the diagonal entries are monomials in the matrix entries $A_{ij}$ with $i + j = N+1$ and $(i,j) \neq (n+1,n+1)$, along with the matrix entries $B_{ij}$ with $i+j = N$ (note that these are precisely the matrix entries dividing the $Q$-invariant, and that none of them are zero). Thus, $M(A^\circ, B^\circ)$ is invertible over $\Q_p$, and there exists a column vector $C(A^\circ,B^\circ) \in \Z_p^{N+1}$ such that
$$M(A^\circ,B^\circ) \cdot {\small\left[\begin{array}{ccccccc} A_{(n+1)(n+1)} & B_{(n+1)(n+1)} & A_{(n+2)(n+2)} & B_{(n+2)(n+2)} & \cdots & A_{NN} & B_{NN} \end{array}\right]^T} + C(A^\circ,B^\circ)$$
is the column vector whose entries are the coefficients of $\on{inv}(A^\circ,B^\circ)$. Consequently, we arrive at the following result:

\begin{lemma} \label{lem-secs}
Let $(A^\circ, B^\circ)$ be fixed as above. Given $f \in U_N(\Z_p)$, there exists a unique tuple $t \in \Q_p^{N+1}$ such that $\on{inv}(A^\circ,B^\circ)(t) = f$. Moreover, for any $G_N(\Z_p)$-invariant subset $\mathfrak{S}_p \subset W_N^0(\Z_p)$ that is the preimage under reduction modulo $p^j$ of a nonempty subset of $W_N^0(\Z/p^j\Z)$ for some $j > 0$, the set $\mc{U}(A^\circ,B^\circ)$ of forms $f \in U_N(\Z_p)$ such that $(A^\circ,B^\circ)(t) \in \mc{F}_{\vec{a},\vec{b}}(p) \cap \mathfrak{S}_p$, where $t \in \Z_p^{N+1}$ is the tuple corresponding to $f$, is the closure of a nonempty open subset. In fact, $\mc{U}(A^\circ,B^\circ)$ is defined by congruence conditions modulo a power of $p$ depending only on $\vec{a}$, $\vec{b}$, $n$, and $j$.
\end{lemma}
%

Proposition~\ref{prop-padicfund2} and Lemma~\ref{lem-secs} 
imply the following result, which plays a crucial role in the proof of Theorem~\ref{thm-main2} (see \S\ref{sec-themain2}):
\begin{prop} \label{lem-loco}
    For any $\mathfrak{S}_p$ as in Lemma~\ref{lem-secs}, the function that sends $f \in U_N(\Z_p)$ to the number of $G_N(\Z_p)$-equivalence classes of pairs in $\on{inv}^{-1}(f) \cap \mc{W}_{\vec{a},\vec{b}}(p) \cap \mathfrak{S}_p$ is locally constant. This function is defined by congruence conditions modulo a power of $p$ depending only on $\vec{a}$, $\vec{b}$, $n$, and $\mathfrak{S}_p$, and is also absolutely bounded by \mbox{such a power of $p$.}
\end{prop}
\begin{proof}
We first prove local constancy. By Proposition~\ref{prop-padicfund2}, it suffices to show that the function that sends $f \in U_N(\Z_p)$ to the number of pairs $(A,B) \in \mc{F}_{\vec{a},\vec{b}}(p) \cap \mathfrak{S}_p$ with $\on{inv}(A,B) = f$ is locally constant.
    So take $f \in U_N(\Z_p)$. Among the pairs $(A^\circ, B^\circ)$ constructed above, let $S_f$ be the subset of pairs such that $f \in \mc{U}(A^\circ, B^\circ)$, and let $\ol{S}_f$ be the complement of $S_f$. Let $\mc{U}_1 =  \bigcup_{(A^\circ, B^\circ) \in \ol{S}_f} \mc{U}(A^\circ, B^\circ)$, and let $\mc{U}_2= \bigcap_{(A^\circ, B^\circ) \in S_f} \mc{U}(A^\circ,B^\circ)$. Then for any $g \not\in \mc{U}_1$ (resp., $g \in \mc{U}_2$), the number of pairs $(A,B) \in \mc{F}_{\vec{a},\vec{b}}(p) \cap \mathfrak{S}_p$ with $\on{inv}(A,B) = g$ is at most (resp., at least) $\#S_f$. Hence, for any $g \in \mc{U}_2 - \mc{U}_1$, the number of pairs $(A,B) \in \mc{F}_{\vec{a},\vec{b}}(p)\cap \mathfrak{S}_p$ with $\on{inv}(A,B) = g$ is equal to $\#S_f$. This establishes the desired local constancy, as each $\mc{U}(A^\circ, B^\circ)$ is the closure of a nonempty open subset, so $\mc{U}_2 - \mc{U}_1$ is an open neighborhood of $f$.

    That the function is defined by congruence conditions modulo a power of $p$ depending only on $\vec{a}$, $\vec{b}$, and $n$ follows from the fact that this holds for each of the sets $\mc{U}(A^\circ, B^\circ)$, and hence also for the set $\mc{U}_2 - \mc{U}_1$. Finally, the claimed bound is an immediate corollary of Proposition~\ref{prop-padicfund2} and Lemma~\ref{lem-secs}.
\end{proof}

\subsection{Proof of Theorem~\ref{thm-stronglocglob}} \label{sec-rep}

\noindent We claim that the integral orbits of $G_N$ on $W_N^0$ satisfy the desired local-to-global principle as long as the following four properties hold:
\begin{enumerate}
    \item[(1)] The algebraic group $G_N$ has class number $1$ over $\Q$;
  \item[(2)] For every $f \in U_N(\C)$ with nonzero discriminant, the set $\on{inv}^{-1}(f) \cap W_N^0(\C)$ is nonempty and consists of a single $G_N(\C)$-orbit; and
  \item[(3)] For every $f \in U_N(\C)$ with nonzero discriminant, each element of $\on{inv}^{-1}(f) \cap W_N^0(\C)$ has trivial stabilizer in $G_N(\C)$.
  \item[(4)] For every $f \in U_N(\Z)$ with nonzero discriminant, the set $\on{inv}^{-1}(f) \cap W_N^0(\Z)$ is nonempty.
\end{enumerate}
This claim follows from~\cite[Theorem~22]{cuspy}, which is a general theorem giving criteria under which the orbits of a finite-dimensional representation of an algebraic group satisfy a local-to-global principle. 
Note that the first three properties above have already been verified in Propositions~\ref{prop-classone} and~\ref{prop-ver2}, and the fourth property follows immediately from the existence of the explicit section $\sigma_0$. This completes the proof of Theorem~\ref{thm-stronglocglob}. \hspace*{\fill}\qed

\medskip

For the sake of concreteness, we now give a self-contained proof of Theorem~\ref{thm-stronglocglob}. Given a principal ideal domain $R$ with fraction field $K$ and an element $w \in W_N^0(R)$ having nonzero $Q$-invariant, write $G_N(K)_{w} \defeq \{g \in G_N(K) : g \cdot w \in G_N(R)\}$. Then the set of $G_N(R)$-orbits contained in the $G_N(K)$-orbit of $w$ is in bijection with the double coset space $G_N(R) \backslash G_N(K)_{w}/\on{Stab}_{G_N(K)}(w) = G_N(R) \backslash G_N(K)_{w}$, where the last step follows from Proposition~\ref{prop-ver2}.

Now, fix $w_0 \in W_N^0(\Z)$ with $Q(w_0) \neq 0$, and suppose for each prime $p$ we have $w_p \in W_N^0(\Z_p)$ with $\on{inv}(w_p) = f$. Our goal is to construct an element $w \in W_N^0(\Z)$, unique up to the action of $G_N(\Z)$, that is $G_N(\Z_p)$-equivalent to $w_p$ for each prime $p$. To do this, consider the diagonal embedding $G_N(\Q) \hookrightarrow \prod_p G_N(\Q_p)$. We claim that this embedding induces a bijection
\begin{equation} \label{eq-locglobbij}
G_N(\Z) \backslash G_N(\Q)_{w_0} \longrightarrow \prod_p G_N(\Z_p) \backslash G_N(\Q_p)_{w_0}.
\end{equation}
Note that the product on the right-hand side of~\eqref{eq-locglobbij} is in fact a finite product, because if $p$ is a prime such that $G_N(\Z_p) \backslash G_N(\Q_p)_{w_0} \neq 1$, then by Proposition~\ref{prop-singleQ}, we must have $p \mid Q(w_p) \mid \on{disc}(f)$. Verifying injectivity of the map in~\eqref{eq-locglobbij} is easy: if $g_1, g_2 \in G_N(\Q)_{w_0}$ have the same image, then $g_1g_2^{-1} \in G_N(\Q) \cap \bigcap_p G_N(\Z_p) = G_N(\Z)$. As for surjectivity, if $(g_p)_p \in \prod_p G_N(\Z_p) \backslash G_N(\Q_p)_{w_0}$, then since $G_N$ has class number $1$ over $\Q$ (by Proposition~\ref{prop-classone}), there exists $g \in G_N(\Q)$ such that $g$ maps to $g_p$ under the map $G_N(\Q) \to G_N(\Z_p) \backslash G_N(\Q_p)$; but then $g \cdot w_0 \in W_N^0(\Q) \cap \bigcap_p W_N^0(\Z_p) = W_N^0(\Z)$, implying that $g \in G_N(\Q)_{w_0}$.


By Proposition~\ref{prop-ver2}, which implies that the $G_N(\Q_p)$-orbit of $w_0$ is equal to that of $w_p$ for each prime $p$, we may view the tuple $(w_p)_p$ as an element of the right-hand side of~\eqref{eq-locglobbij}. Then, under the bijection, the tuple $(w_p)_p$ corresponds to the $G_N(\Z)$-orbit of the desired element $w \in W_N^0(\Z)$, as necessary.  \hspace*{\fill}\qed

\medskip

As an immediate consequence of Theorem~\ref{thm-stronglocglob} along with Proposition~\ref{prop-singleQ}, we have the following result concerning global integral orbits having unit $Q$-invariant:
\begin{corollary} \label{cor-oneandonly}
Let $f \in U_N(\Z)$. Then the set $\{(A,B) \in \on{inv}^{-1}(f) \cap W_N^0(\Z) : |Q(A,B)| = 1\}$ is nonempty and consists of a single $G_N(\Z)$-orbit.
\end{corollary}

\section{Asymptotic formulae for the count of reducible orbits} \label{sec-pfthmain2}

In this section, we use the local-to-global principle in Theorem~\ref{thm-stronglocglob} to deduce Theorem~\ref{thm-main2}, which gives an asymptotic formula for the count of reducible $\on{SL}_N(\Z)$-orbits on $W_N(\Z)$ in terms of a product of local integrals. We then perform a change-of-variables argument to rewrite each of these integrals in a more convenient form, thus proving Theorem~\ref{thm-main3}.

\subsection{Proof of Theorem~\ref{thm-main2}} \label{sec-themain2}

Let $\mathfrak{S}$ be a big family in $W_N^0(\Z)$. We start by proving the following asymptotic formula for the count of $G_N(\Z)$-orbits on $\mathfrak{S}$ of bounded height.
\begin{theorem} \label{thm-acceptcubic2}
The number of $G_N(\Z)$-orbits on $\mathfrak{S}$ of height up to $X$ is given by
\begin{equation} \label{eq-midacceptcubic2}
 \on{N}_N^{(r)}(X) \times \prod_p \int_{f \in U_N(\Z_p) } \#\left(\frac{\on{inv}^{-1}(f)\cap \mathfrak{S}_p}{G_N(\Z_p)}\right) df + o\big(X^{N+1}\big).
\end{equation}
Moreover, when $N = 3$, the number of $(\on{GL}_2 \times G_3)(\Z)$-orbits on $\mathfrak{S}$ with discriminant having absolute value less than $X$ is given by
\begin{equation} \label{eq-midacceptcubic2GL2}
 \on{N}_\Delta^{(r)}(X) \times \prod_p \int_{f \in U_3(\Z_p) } \#\left(\frac{\on{inv}^{-1}(f)\cap \mathfrak{S}_p}{G_3(\Z_p)}\right) df + o(X),
\end{equation}
where $\on{N}_\Delta^{(r)}(X)$ is the number of $\on{GL}_2(\Z)$-orbits of irreducible binary cubic forms in $U_3(\Z)^{(r)}$ of discriminant up to $X$ in absolute value.
\end{theorem}
\begin{remark}
The techniques for counting orbits developed in the works of Bhargava et al.~work systematically for quite general representations of reductive groups. Theorem~\ref{thm-acceptcubic2} constitutes a rare example of a situation in which we can determine precise asymptotics for the integral orbits of the action of a non-reductive group (namely, $G_N$).
\end{remark}
\begin{proof}[Proof of Theorem~\ref{thm-acceptcubic2}]
We first prove~\eqref{eq-midacceptcubic2}, and then we explain how the proof of~\eqref{eq-midacceptcubic2GL2} differs. The proof is analogous to that of~\cite[Theorem 24]{cuspy}, but we include the details for the sake of completeness. Fix an integer $\mathfrak{b} \geq 1$, and factorize it into primes as $\mathfrak{b} = \prod_p p^{e_p}$. We start by proving an analogue of Theorem~\ref{thm-acceptcubic2} with $\mathfrak{S}$ replaced by the subfamily $\mathfrak{S}(\mathfrak{b}) \defeq \{w \in \mathfrak{S} : {\vert}\lambda(w){\vert} = \mathfrak{b}\}$; note that $\mathfrak{S}(\mathfrak{b})$ is itself a big family in $W_N^0(\Z)$, where $\mathfrak{S}(\mathfrak{b})_{p} = \{w \in \mathfrak{S}_p : {\vert}\lambda(w){\vert}_p = {\vert}\mathfrak{b}{\vert}_p\}$.

If $\mathfrak{S}(\mathfrak{b}) = \varnothing$, then there is nothing to prove, so assume that $\mathfrak{S}(\mathfrak{b}) \neq \varnothing$. For each prime $p \mid \mathfrak{b}$, we partition $U_N(\Z_p)$ as $U_N(\Z_p) = \bigsqcup_{j = 1}^{m_p} U_{p,j}$, where each $U_{p,j}$ is a level set for the function that sends $f \in U_N(\Z_p)$ to $\#(G_N(\Z_p) \backslash (\on{inv}^{-1}(f) \cap \mathfrak{S}(\mathfrak{b})_p))$. Write ``$E(m)$'' to mean ``a power of $m$ that depends only on $n$ and $\mathfrak{S}$.'' Then Proposition~\ref{lem-loco} implies that $U_N(\Z_p)$ can be covered by open sets, each of which is defined by congruence conditions modulo $E(p^{e_p})$, such that this orbit-counting function is constant on each open. It follows that $U_{p,j}$ is defined by congruence conditions modulo $E(p^{e_p})$. The quantity $\#(G_N(\Z_p) \backslash (\on{inv}^{-1}(f) \cap \mathfrak{S}(\mathfrak{b})_p))$ is independent of the choice of $f \in U_{p,j}$ (by the definition of a level set) and by Proposition~\ref{lem-loco}, 
this quantity is $\ll E(p^{e_p})$. 

Now for each prime $p \nmid \mathfrak{b}$, Proposition~\ref{prop-singleQ} tells us that $\#(G_N(\Z_p) \backslash (\on{inv}^{-1}(f) \cap \mathfrak{S}(\mathfrak{b})_p)) = 1$ for each $f \in \on{inv}(\mathfrak{S}(\mathfrak{b})_p)$. It then follows from Theorem~\ref{thm-stronglocglob} that the quantity
\begin{equation} \label{eq-thistheprod}
\#\left(\frac{\on{inv}^{-1}(f) \cap \mathfrak{S}(\mathfrak{b})}{G_N(\Z)}\right) = \prod_p \#\left(\frac{\on{inv}^{-1}(f) \cap \mathfrak{S}(\mathfrak{b})_p}{G_N(\Z_p)}\right) =\prod_{p \mid \mathfrak{b}} \#\left(\frac{\on{inv}^{-1}(f) \cap \mathfrak{S}(\mathfrak{b})_p}{G_N(\Z_p)}\right)
\end{equation}
is independent of the choice of $f \in \on{inv}(\mathfrak{S}(\mathfrak{b})) \cap \bigcap_p U_{p,j_p}$ for each tuple $(j_p)_{p \mid \mathfrak{b}} \in \prod_{p \mid \mathfrak{b}} \{1,\dots, m_p\}$. Therefore, we have
\begin{equation} \label{eq-fixedfquad}
\sum_{\substack{f \in U_N(\Z) \,\cap\, \bigcap_p U_{p,j_p} \\ \on{H}(f) < X }} \#\left(\frac{\on{inv}^{-1}(f) \cap \mathfrak{S}(\mathfrak{b})}{G_N(\Z)}\right) = \#\left(\frac{\on{inv}^{-1}(f^*) \cap \mathfrak{S}(\mathfrak{b})}{G_N(\Z)}\right) \times \sum_{\substack{f \in \on{inv}(\mathfrak{S}(\mathfrak{b})) \,\cap\, \bigcap_p U_{p,j_p} \\ \on{H}(f) < X }} 1
\end{equation}
where $f^* \in \on{inv}(\mathfrak{S}(\mathfrak{b})) \,\cap\, \bigcap_{p \mid \mathfrak{b}} U_{p,j_p}$ is any fixed element. Since $\mathfrak{S}$ is a big family, and since the aforementioned explicit section $\sigma_0$ has image contained in the locus of pairs $(A,B)$ with $\lambda(A,B) = \pm 1$, it follows that $\on{inv}(\mathfrak{S}(\mathfrak{b})_p) = U_N(\Z_p)$ for every $p \gg 1$ that does not divide $\mathfrak{b}$. 
As the set $\on{inv}(\mathfrak{S}(\mathfrak{b})) \,\cap\, \bigcap_p U_{p,j_p}$ is defined by congruence conditions modulo $E(\mathfrak{b})$, since $\on{inv}(\mathfrak{S}(\mathfrak{b})_p) \cap U_{p,j_p}$ is defined by congruence conditions modulo $E(p^{e_p})$ for each $p$, 
we obtain the following asymptotic:
\vspace*{0.1in}
\begin{equation} \label{eq-fixedfquad2}
\displaystyle \sum_{\substack{f \in \on{inv}(\mathfrak{S}(\mathfrak{b})) \,\cap\, \bigcap_p U_{p,j_p} \\ \on{H}(f) < X }} 1 = \displaystyle \on{N}_N^{(r)}(X) \times \prod_{p \mid \mathfrak{b}} \int_{f \in \on{inv}(\mathfrak{S}(\mathfrak{b})_p) \cap U_{p,j_p}} df \times \prod_{p \nmid \mathfrak{b}} \int_{f \in \on{inv}(\mathfrak{S}(\mathfrak{b})_p)} df + O\big(E(\mathfrak{b})X^{N+1 - \delta}\big)
\end{equation}
for some sufficiently small $\delta > 0$. Substituting the asymptotic~\eqref{eq-fixedfquad2} into the right-hand side of~\eqref{eq-fixedfquad}, applying~\eqref{eq-thistheprod} to the resulting expression, and summing that over tuples $(j_p)_{p \mid \mathfrak{b}} \in \prod_{p \mid \mathfrak{b}} \{1,\dots, m_p\}$ gives the following:
\begin{align}\label{eq-fixedconductoresult}
\displaystyle \sum_{\substack{f \in U_N(\Z)\\ \on{H}(f) < X }}  \#\left(\frac{\on{inv}^{-1}(f) \cap \mathfrak{S}(\mathfrak{b})}{G_N(\Z)}\right) & =  \displaystyle \on{N}_N^{(r)}(X) \times \prod_{p} \int_{f \in U_N(\Z_p)} \#\left(\frac{\on{inv}^{-1}(f) \cap \mathfrak{S}(\mathfrak{b})_p}{G_N(\Z_p)}\right) df \\
& \qquad\qquad\qquad\qquad\qquad\qquad\qquad +O\big(E(\mathfrak{b})X^{N+1 - \delta}\big). \nonumber
\end{align}

Next, we prove that the theorem holds with ``$=$'' replaced by ``$\geq$.'' For any real number $M > 1$, let $\mathfrak{S}[M] \defeq \{w \in \mathfrak{S} : {\vert}\lambda(w){\vert} < M\}$. Summing~\eqref{eq-fixedconductoresult} over $\mathfrak{b} < M$ gives:
\begin{align} \label{eq-conductorboundquad}
 \sum_{\substack{f \in U_N(\Z)\\ \on{H}(f) < X }} \displaystyle  \#\left(\frac{\on{inv}^{-1}(f) \cap \mathfrak{S}[M]}{G(\Z)}\right)  & = \on{N}_N^{(r)}(X) \times \sum_{\mathfrak{b} < M} \prod_{p} \int_{f \in U_N(\Z_p)} \#\left(\frac{\on{inv}^{-1}(f) \cap \mathfrak{S}(\mathfrak{b})_p}{G(\Z_p)}\right) df \\
 & \qquad\qquad\qquad\qquad\qquad\qquad\qquad +O\big(E(M)X^{N+1 - \delta}\big).\nonumber
\end{align}
Dividing through by $\on{N}_N^{(r)}(X)$, letting $X \to \infty$, and replacing $\mathfrak{S}$ with $\mathfrak{S}[M]$, it follows \mbox{from~\eqref{eq-conductorboundquad} that}
\begin{equation} \label{eq-conductorboundquad2}
\liminf_{X \to \infty} \frac{\sum_{\substack{f \in U_N(\Z) \\ \on{H}(f) < X }}  \displaystyle \#\left(\frac{\on{inv}^{-1}(f) \cap \mathfrak{S}}{G_N(\Z)}\right) }{\on{N}_N^{(r)}(X)} \geq \sum_{\mathfrak{b} < M}\prod_{p } \int_{f \in U_N(\Z_p)} \#\left(\frac{\on{inv}^{-1}(f) \cap \mathfrak{S}(\mathfrak{b})_p}{G_N(\Z_p)}\right)  df.
\end{equation}
Now, letting $M \to \infty$ on the right-hand side of~\eqref{eq-conductorboundquad2} and factoring the sum into an Euler product, we obtain the following:
\begin{align} 
     \sum_{\mathfrak{b} = 1}^\infty\prod_{p } \int_{f \in U_N(\Z_p)} \#\left(\frac{\on{inv}^{-1}(f) \cap \mathfrak{S}(\mathfrak{b})_p}{G_N(\Z_p)}\right)  df & = \prod_p \sum_{e = 0}^\infty \int_{f \in U_N(\Z_p)} \#\left(\frac{\on{inv}^{-1}(f) \cap \mathfrak{S}(p^e)_p}{G_N(\Z_p)}\right)  df \nonumber \\ & = \prod_p  \int_{f \in U_N(\Z_p)} \#\left(\frac{\on{inv}^{-1}(f) \cap \mathfrak{S}_{p}}{G_N(\Z_p)}\right)  df. \label{eq-interchange}
\end{align}
Combining~\eqref{eq-conductorboundquad2} with~\eqref{eq-interchange}, we find that Theorem~\ref{thm-acceptcubic2} holds with ``$=$'' replaced by ``$\geq$.''

It thus remains to prove the theorem with ``$=$'' replaced by ``$\leq$.'' Let $\mathfrak{S}[M]' \defeq \mathfrak{S} \smallsetminus \mathfrak{S}[M]$. Then for each $w \in \mathfrak{S}[M]'$, we have that ${\vert}\lambda(w){\vert} \geq M$. 
In~\cite{sqfrval2}, Bhargava, Shankar, and Wang determined bounds for the number of reducible $\on{SL}_N(\Z)$-orbits on $W_N(\Z)$ having large $\lambda$-invariant (for those orbits on which the notion of $\lambda$-invariant can be extended naturally). In particular, by~\cite[(14), (16), and Theorem~4.1]{sqfrval2} along with Proposition~\ref{prop-sasymp} (to follow), we can choose $\delta \in (0,1)$ so that
\begin{align}
& \sum_{\substack{f \in U_N(\Z) \\ \on{H}(f) < X }}  \#\left(\frac{\on{inv}^{-1}(f) \cap \mathfrak{S}[M]'}{G_N(\Z)}\right) =O_\varepsilon\big(X^{N+1+\varepsilon}/M\big) + O\big(X^{N+1 - \delta}\big). \label{eq-discboundquad}
\end{align}
On the other hand, it follows from~\eqref{eq-conductorboundquad} that
\begin{align} \label{eq-conductorboundquad3}
\displaystyle 
\sum_{\substack{f \in U_N(\Z) \\ \on{H}(f) < X }}  \#\left(\frac{\on{inv}^{-1}(f) \cap \mathfrak{S}[M]}{G_N(\Z)}\right) &\leq \displaystyle \on{N}_N^{(r)}(X) \times \prod_{p} \int_{f \in U_N(\Z_p)} \#\left(\frac{\on{inv}^{-1}(f) \cap \mathfrak{S}_{p}}{G_N(\Z_p)}\right) df \\
& \qquad\qquad\qquad\qquad\qquad\quad\,\, + O\big(E(M)X^{N+1 - \delta}\big). \nonumber
\end{align}
Taking $M$ to grow as a sufficiently small power of $X$ and combining~\eqref{eq-discboundquad} with~\eqref{eq-conductorboundquad3} yields~\eqref{eq-midacceptcubic2}.

As for the proof of~\eqref{eq-midacceptcubic2GL2}, the only steps that differ are the deductions of~\eqref{eq-fixedfquad2} and~\eqref{eq-discboundquad}. For~\eqref{eq-fixedfquad2}, one simply applies the asymptotics for counting $\on{GL}_2(\Z)$-orbits of binary cubic forms satisfying local conditions obtained by Bhargava, Shankar, and Tsimerman in~\cite[Theorem~26]{MR3090184}. For~\eqref{eq-discboundquad}, one simply applies the estimate proven by Bhargava in~\cite[Proposition~23]{MR2183288}.
\end{proof}
To deduce Theorem~\ref{thm-main2} from Theorem~\ref{thm-acceptcubic2}, we require the following result relating $G_N(\Z)$-orbits on $W_N^0(\Z)$ with reducible $\on{SL}_N(\Z)$-orbits on $W_N(\Z)$:

\begin{prop} \label{prop-sasymp}
Let $f \in U_N(\Z)$ be irreducible. The $G_N(\Z)$-orbits on $\on{inv}^{-1}(f) \cap W_N^0(\Z)$ are in bijection with the reducible $\on{SL}_N(\Z)$-orbits on $\on{inv}^{-1}(f) \cap W_N(\Z)$.


Analogously, let $\Delta \in \Z \smallsetminus \{0\}$. The $(\on{GL}_2 \times G_3)(\Z)$-orbits on $\on{disc}^{-1}(\Delta) \cap W_3^0(\Z)$ are in bijection with the reducible $(\on{GL}_2 \times \on{SL}_3)(\Z)$-orbits on $\on{disc}^{-1}(\Delta) \cap W_3(\Z)$.
\end{prop}
\begin{proof}
We first claim that if we have $(A_1, B_1), (A_2, B_2) \in \on{inv}^{-1}(f) \cap W_N^0(\Z)$ and $g_1 \in \on{SL}_N(\Z)$ such that $g_1 \cdot (A_1, B_1) = (A_2,B_2)$, then $g_1 \in G_N(\Z)$. Indeed, by Proposition~\ref{prop-ver2}, there exists $g_2 \in G_N(\Q)$ such that $g_2 \cdot (A_1, B_1) = (A_2,B_2)$, so $g_1^{-1}g_2 \in \on{SL}_N(\Q)$ stabilizes $(A_1, B_1)$. But by~\cite[Proposition~14]{Swpreprint}, the stabilizer in $\on{SL}_N(\Z)$ of any element of $\on{inv}^{-1}(f)$ is trivial, so $g_1 = g_2 \in \on{SL}_N(\Z) \cap G_N(\Q) = G_N(\Z)$, as claimed. 

It follows from the above claim that there are at least as many reducible $\on{SL}_N(\Z)$-orbits on $\on{inv}^{-1}(f) \cap W_N(\Z)$ as there are $G_N(\Z)$-orbits $\on{inv}^{-1}(f) \cap W_N^0(\Z)$. It therefore suffices to show that every reducible $\on{SL}_N(\Z)$-orbit on $W_N(\Z)$ meets the linear subspace $W_N^0(\Z)$, but this was proven in~\cite[\S3.4, bottom of p.~12]{sqfrval2}.

The proof of the analogous claim about $(\on{GL}_2 \times G_3)(\Z)$-orbits on $\on{disc}^{-1}(\Delta) \cap W_3^0(\Z)$ is essentially identical, with~\cite[Proposition~14]{Swpreprint} replaced by~\cite[\S2.1, p.~1039]{MR2183288}.
\end{proof}

 Theorem~\ref{thm-main2} now follows from Theorem~\ref{thm-acceptcubic2} and Proposition~\ref{prop-sasymp} by taking $\mathfrak{S} = S \cap W_N^0(\Z)$, where $S$ is a big family in $W_N(\Z)$.

\subsection{Change-of-variables formulas, and proof of Theorem~\ref{thm-main3}}

The purpose of this section is to deduce Theorem~\ref{thm-main3} from Theorem~\ref{thm-main2}. The objective is, for each prime $p$, to reexpress the integral over $U_N(\Z_p)$ that occurs in the asymptotic given by Theorem~\ref{thm-main2} in terms of an integral over $W_N^0(\Z_p)$. To do this, we first write the integral over $U_N(\Z_p)$ as an integral over $G_N(\Z_p) \times U_N(\Z_p)$; we then perform a change-of-variables to relate the natural measure on $G_N \times U_N$ with the natural measure on $W_N^0$ (note that $\dim G_N \times U_N = \dim G_N + \dim U_N = \dim W_N^0$). 

We also prove a second change-of-variables formula relating the natural measure on $W_N^{\on{top}}$ with the natural measure on $(\on{SL}_n \times \on{SL}_{n+1}) \times \mathbb{A}^1$, where $\mathbb{A}^1$ parametrizes the $Q$-invariant (note that $\dim \big((\on{SL}_n \times \on{SL}_{n+1}) \times \mathbb{A}^1\big) = \dim \on{SL}_n + \dim \on{SL}_{n+1} + \dim \mathbb{A}^1 = \dim W_N^{\on{top}}$). This second change-of-variables plays a crucial role in the proof of Theorem~\ref{thm-main1} (see \S\ref{sec-pfthmn1}, to follow).

\subsubsection{Explicit choices of volume forms} \label{sec-rightHaar}

We start by making explicit choices of volume forms on $U_N$, $W_N^0$, and $G_N$. Take $df$, $dw$, $dh_1$, and $dh_2$ to be generators of the $\Z$-modules of right-invariant volume forms on $U_N$, $W_N^0$, $H_1$, and $H_2$, respectively, all defined over $\Z$. On $G_N$ we take the measure $dg = dh_1dh_2$ (recall that $G_N$ is the Frobenius product of $H_1$ and $H_2$ by Lemma~\ref{lem-prodhgl2}).

We now give an explicit formula for $dg$ on an open subscheme of $G_N$. For a square matrix $M$, we denote the minor obtained by deleting the first row and column by $M_{(1,1)}$; if $M$ is $1$-dimensional, then we set $M_{(1,1)} \defeq 1$. Let $G_N^\circ$ be the open subscheme of $G_N$ whose $R$-points are given by matrices
$$g = \left[\begin{array}{c|c} g' & 0 \\ \hline g''' & g'' \end{array}\right] \in G_N(R)$$
such that $g'_{(1,1)} \in R^\times$ for any $\Z$-algebra $R$. Then we may realize $G_N^\circ$ as an open subscheme of the affine space
$$\mc{M} \defeq \on{Spec} \Z\left[\left\{M_{ij} : \begin{array}{c} 1 \leq i,j \leq n, \text{ where} \\ (i,j) \neq (1,1) \text{ and } j \leq n \text{ if } i \leq n  \end{array}\right\}\right]$$
via the map that sends a matrix $g \in G_N(R)$ to its list of matrix entries $g_{ij}$, excluding the entries $g_{ij}$ for any $(i,j)$ such that $i \leq n$ and $j \geq n+1$ or $(i,j) = (1,1)$.

Let $\prod dM_{ij}$ be the Haar measure on $\mc{M}$, normalized so that $\mc{M}(\Z)$ has covolume $1$ in $\mc{M}(\R)$. We denote by $\prod dg_{ij}$ the restriction of this measure to $G_N^\circ$ via the embedding $G_N^\circ \subset \mc{M}$ defined above, and by abuse of notation, we also denote by $dg$ the restriction to $G_N^\circ$ of the volume form on $G_N$. Then a calculation shows that the measure $(g'_{(1,1)})^{-1}(\det g'')^{n-1} \times \prod dg_{ij}$ is invariant under the right-action of $G_N$, and so on $G_N^\circ$ we have $dg = (g'_{(1,1)})^{-1}(\det g'')^{n-1} \times \prod dg_{ij}$, up to sign. As for the left-action of $G_N$, if we take
\begin{equation} \label{eq-defhsplit}
h = \left[\begin{array}{c|c} h' & 0 \\ \hline h''' & h'' \end{array}\right] \in G_N(R),
\end{equation}
then one readily verifies that
\begin{equation} \label{eq-defrhoh}
d(hg) = \rho(h)dg, \quad \text{where} \quad \rho(h) \defeq |\det h''|^{N}.
\end{equation}
Note that the formula~\eqref{eq-defrhoh} holds on all of $G_N$, not just on the open subscheme $G_N^\circ$.

We finish by making explicit choices of the volume forms on $\mathbb{A}^1$ and $\on{SL}_n \times \on{SL}_{n+1}$. Take $dq$ to be any volume form on $\mathbb{A}^1$ defined over $\Z$, and take $dh$ to be any volume form on $\on{SL}_n \times \on{SL}_{n+1}$ defined over $\Z$. The forms $dq$ and $dh$ are necessarily left- and right-invariant.

\subsubsection{Stating the change-of-variables formulas}

We are now in position to state our change-of-variables formula relating the measure $dw$ on $W_N^0$ with the measure $dgdf$ on $G_N \times U_N$.

\begin{prop} \label{prop-jac}
Let $R=\R$ or $\Z_p$ for a prime $p$. Let $\phi \colon W_N^0(R) \to \R$ be an integrable function. Then there exists a nonzero rational number $\mc{J} \in \Q^\times$, possibly depending on $N$, such that
\begin{equation*}
\frac{1}{|\mc{J}|}  \int_{w\in W_N^0(R)}\phi(w)|\lambda(w)|dw = \int_{\substack{f\in U_N(R)\\\Delta(f)\neq 0}} \sum_{[w]\in\frac{\inv^{-1}(f) \cap W_N^0(R)}{G_N(R)}} \int_{g \in G_N(R)}
\phi(g \cdot w)dgdf,
\end{equation*}
where $|-|$ denotes the usual absolute value on $R$.
\end{prop}
\begin{proof}
We follow the general strategy used in the proof of~\cite[Proposition~14]{cuspy}. Take $R = \R$, and let $\mc{U}\subset U_N(\R)$ be an open set and let $\sigma\colon \mc{U}\to W_N^0(\R)$ be a continuous section of $\inv$ (note that such a section exists, namely the aforementioned explicit section $\sigma_0$). We first claim that for some $\mc{J} \in \Q^\times$ we have
\begin{equation}\label{eq-jac1}
\int_{w \in G_N(\R) \cdot \sigma(\mc{U})} \phi(w)|\lambda(w)| dw  = |\mc{J}|\int_{f \in \mc{U}} \int_{g\in G_N(\R)} \phi(g \cdot \sigma(f)) dgdf.
\end{equation}
By the Stone--Weierstrass theorem, it suffices to treat the case where $\sigma$ is piecewise analytic. In that case, we have
\begin{equation*}
\int_{w \in G_N(\R) \cdot \sigma(\mc{U})} \phi(w)|\lambda(w)| dw  = \int_{f \in \mc{U}} \int_{g\in G_N(\R)}|\mc{J_\sigma}(g,f)| \phi(g\cdot \sigma(f))dgdf,
\end{equation*}
where $\mc{J}_{\sigma}(g,f)$ is the determinant of the Jacobian matrix coming from the change-of-variables taking the measure $\lambda(w)dw$ on $W_N^0$ to the product measure $dgdf$ on $G_N \times U_N$. 

We now show that
$\mc{J}_{\sigma}(g,f)$ is independent of $g$. Take $h \in G_N(\R)$, and consider the transformation on $W_N^0(\R)$ that sends $w \mapsto h \cdot w$. Then there exists a function $\chi_\lambda \colon G_N(\R) \to \R_{> 0}$ such that $\lambda(h \cdot w)d(h\cdot w) = \chi_\lambda(h) \lambda(w) dw$; indeed, one checks that if $h$ is expressed as in~\eqref{eq-defhsplit}, we have $\lambda(h \cdot w) = |\det h''|^{-1} Q(w)$ and $d(h \cdot w) = |\det h''|^{N+1} \cdot dw$, so $\chi_\lambda(h)= |\det h''|^N$. 
On the other hand, the transformation $w \mapsto h \cdot w$ acts on $G_N(\R) \times \mc{U}$ by sending $(g,f) \mapsto (h g,f)$. Letting $\rho \colon G_N(\R) \to \R_{>0}$ be as in~\eqref{eq-defrhoh}, we have that
\begin{equation*}
\mc{J}_{\sigma}(h  g,f) d(h  g) df = \rho(h)\mc{J}_{\sigma}(h  g,f) dgdf.
\end{equation*}
But we also have that
\begin{equation*}
\mc{J}_\sigma(h g,f)d(h g)df=\lambda(h \cdot w)d(h \cdot w)=\chi_\lambda(h)\lambda(w)dw=\chi_\lambda(h)\mc{J}_\sigma(g,f)dgdf.
\end{equation*}
Upon comparing the above two displayed equations, and using the fact that $\rho(h)= |\det h''|^{N} = \chi_\lambda(h)$, we deduce that the function $\mc{J}_\sigma(g,f)$ is independent of $g$.

That $\mc{J}_{\sigma}(g,f)$ is independent of $\sigma$ follows from an argument identical to Step 2 in the proof of~\cite[Proposition 3.10]{MR3272925} (this step requires that the measure $dg$ be right-invariant). Thus, we can take $\sigma$ to be the polynomial section $\sigma_0$. 
With this choice of section, that $\mc{J}_{\sigma_0}(g,f)$ is independent of $f$ and equal to a nonzero rational constant follows from an argument identical to Steps 3 and 4 in the proof of~\cite[Proposition 3.10]{MR3272925}.

We have thus proven \eqref{eq-jac1}. Proposition \ref{prop-jac} --- including the case where $R = \Z_p$ for a prime $p$ --- now follows from \eqref{eq-jac1} and the principle of permanence of identities, just as \cite[Proposition 3.7]{MR3272925}
is deduced from \cite[Proposition 3.10]{MR3272925}.
\end{proof}

We shall also require the following change-of-variables formula relating the pushforward of the measure $dw$ from $W_N^0$ to $W_N^{\on{top}}$ with the product of the Haar measure on $\on{SL}_n \times \on{SL}_{n+1}$ and the measure $dq$ on $\mathbb{A}^1$. The proof is analogous to that of Proposition~\ref{prop-jac}, so we omit it.

\begin{prop} \label{prop-jac2}
Let $p$ be prime. Let $\phi \colon W_N^{\on{top}}(\Z_p) \to \R$ be an integrable function, and extend $\phi$ to a function on $W_N^0$ by precomposing with the natural projection map $W_N^0 \to W_N^{\on{top}}$. Then there exists a nonzero rational number $\mc{J}' \in \Q^\times$, possibly depending on $N$, such that
\begin{equation*}
\int_{w\in W_N^0(\Z_p)} \frac{\phi(w)}{|\lambda(w)|}dw=
|\mc{J}'|_p\times \int_{\substack{q\in \Z_p \smallsetminus \{ 0\}}}\sum_{[w] \in\frac{Q^{-1}(q) \cap W_N^{\on{top}}(\Z_p)}{(\on{SL}_n \times \on{SL}_{n+1})(\Z_p)}}\int_{h \in (\on{SL}_n \times \on{SL}_{n+1})(\Z_p)}
\phi(h \cdot w)dhdq.
\end{equation*}
\end{prop}
We now compute the constants $|\mc{J}|$ and $|\mc{J}'|$ and show, in particular, that they do not actually depend on $N$:
\begin{lemma} \label{lem-ksiolajidebt}
We have that $|\mc{J}| = |\mc{J}'|$ = 1.
\end{lemma}
\begin{proof}
We first prove that $|\mc{J}/\mc{J}'| = 1$. By Corollary~\ref{cor-oneandonly}, there is exactly one $G_N(\Z)$-orbit on $W_N^0(\Z)$ with unit $Q$-invariant lying above a binary form $f \in U_N(\Z)$. In particular, the average number of $G_N(\Z)$-orbits on $W_N^0(\Z)$ with unit $Q$-invariant is equal to $1$. On the other hand, combining Theorem~\ref{thm-acceptcubic2} with Propositions~\ref{prop-jac} and~\ref{prop-jac2} and Lemma~\ref{lem-gvol} and using the formula~\eqref{eq-vollkwhatsit}, we find that the average number of $G_N(\Z)$-orbits on $W_N^0(\Z)$ with unit $Q$-invariant is equal to
\begin{equation*}
|\mc{J}| \times \prod_p \frac{1}{\on{Vol}(G_N(\Z_p))}\int_{w \in \mc{L}_{\vec{0},\vec{0}}(p)} dw = \left|\frac{\mc{J}}{\mc{J}'}\right|.
\end{equation*}
Next, to compute ${\vert}\mc{J}'{\vert}$, it suffices to compute ${\vert}\cJ'{\vert}_p$ for each $p$, because $\mc{J}' \in \Q^\times$. To do this, we construct convenient sets in $W_N^{\on{top}}(\Z_p)$ and compute their volumes in two different ways: first, using Proposition \ref{prop-jac2}, and second, by means of a point count over $\F_p$. Equating the results of the two volume computations yields then \mbox{the value of ${\vert}\mc{J}'{\vert}_p$.}

To this end, fix $\ol{q} \in \mathbb{F}_p^\times$, and let $\phi_p \colon W_N^{\on{top}}(\Z_p) \to \R$ be the indicator function of the set
$$\Sigma \defeq \left\{(A,B) \in W_N^{\on{top}}(\Z_p) : Q(A,B) \equiv \ol{q}\,(\on{mod}p)\right\}.$$
By Propositions~\ref{prop-castle} and~\ref{prop-padicfund}, the group $(\on{SL}_n \times \on{SL}_{n+1})(\Z_p)$ acts simply transitively on the set of elements in $\Sigma$ having any fixed $Q$-invariant. Hence, from Proposition \ref{prop-jac2}, we obtain
\begin{align} \label{eq-prop28right}
\Vol(\Sigma)=
    & {\vert}\mc{J}'{\vert}_p \times \on{Vol}\big(\on{SL}_n \times \on{SL}_{n+1})(\Z_p)\big) \int_{\substack{q \in \Z_p \\ q\equiv \ol{q}\,(\on{mod}p)}}  dq
   \\
   &  ={\vert}\mc{J}'{\vert}_p \times \on{Vol}\big((\on{SL}_n \times \on{SL}_{n+1})(\Z_p)) \times p^{-1}. \nonumber
\end{align}
On the other hand, Proposition~\ref{prop-castle} implies that the group $(\on{SL}_n \times \on{SL}_{n+1})(\F_p)$ acts simply transitively on the mod-$p$ reduction $\ol{\Sigma}$ of $\Sigma$. Thus, we have
\begin{equation} \label{eq-sigmafsize}
\#\ol{\Sigma} = \#(\on{SL}_n \times \on{SL}_{n+1})(\mathbb{F}_p).
\end{equation}
Since $\Vol(\Sigma)=p^{-\dim W_N^{\on{top}}}\times \#\ol{\Sigma}$, $\Vol((\on{SL}_n \times \on{SL}_{n+1})(\Z_p))=p^{-\dim(\on{SL}_n \times \on{SL}_{n+1})} \times \#(\on{SL}_n \times \on{SL}_{n+1})(\mathbb{F}_p)$, and $1+\dim \on{SL}_n \times \on{SL}_{n+1} =\dim W_N^{\on{top}}$, it follows from \eqref{eq-prop28right} and \eqref{eq-sigmafsize} that ${\vert}\mc{J}'{\vert}_p=1$ for all $p$.
\end{proof}

We now turn our attention to the proof of Theorem~\ref{thm-main3}. For this, let $S$ be a big family in $W_N(\Z)$. An application of Proposition~\ref{prop-jac} with $R = \Z_p$ and with $\phi$ equal to the indicator function of $S_p \cap W_N^0(\Z_p)$, along with Lemma~\ref{lem-ksiolajidebt}, yields that
\begin{equation} \label{eq-theswapp}
  \int_{f \in U_N(\Z_p)} \#\left(\frac{\on{inv}^{-1}(f) \cap S_{p} \cap W_N^0(\Z_p)}{G_N(\Z_p)}\right)df = \frac{1}{\on{Vol}(G_N(\Z_p))}\int_{w \in S_p \cap W_N^0(\Z_p)} |\lambda(w)|_pdw.
\end{equation}
In the next lemma, we determine $\on{Vol}(G_N(\Z_p))$:
\begin{lemma} \label{lem-gvol}
We have that $\on{Vol}(G_N(\Z_p)) = \xi_{p,n}^{-1} = (1-p^{-1})(1 - p^{-n-1}) \times \prod_{i = 2}^{n} (1 - p^{-i})^2$, and also that $\on{Vol}\big((\on{SL}_n \times \on{SL}_{n+1})(\Z_p)\big) = (1 - p^{-n-1}) \times \prod_{i = 2}^n (1 - p^{-i})^2$.
\end{lemma}
\begin{proof}
We prove only the claimed formula for $\on{Vol}(G_N(\Z_p))$, as the other formula can be proven similarly. Since $G_N$ is smooth over $\Z$, we have that $\on{Vol}(G_N(\Z_p)) = \#G_N(\mathbb{F}_p)/p^{\dim G_N}$. It is easy to see that $\dim G_N = 3n^2+3n$, and by Lemma~\ref{lem-prodhgl2}, we have that 
\begin{align*}
\#G_N(\mathbb{F}_p) & = \#H_1(\mathbb{F}_p) \times \#H_2(\mathbb{F}_p) = \#\mathbb{A}^{n^2+n}(\mathbb{F}_p) \times \big(\#\on{SL}_n(\F_p) \times \#\on{GL}_{n+1}(\F_p)\big) \\
& = p^{n^2+n} \times (p-1)^{-1} \times \prod_{i = 0}^{n-1} (p^n - p^i) \times \prod_{i = 0}^{n} (p^{n+1} - p^i). \qedhere
\end{align*}
\end{proof}
Finally, substituting the formula for $\on{Vol}(G_N(\Z_p))$ given by Lemma~\ref{lem-gvol} into~\eqref{eq-theswapp} and combining the result with Theorem~\ref{thm-main2} yields Theorem~\ref{thm-main3}.

\section{Evaluation of local volumes for applications}

In this section, we evaluate the local integrals in Theorem~\ref{thm-main3} in two cases: (1) where $S_p = W_N(\Z_p)$ for each prime $p$; and (2) where $S_p$ is the set of projective elements in $W_3(\Z_p)$ for each prime $p$. As a consequence of (1), we deduce Theorem~\ref{thm-main1}, and as a consequence of (2), we deduce Theorems~\ref{thm-z2z3} and~\ref{thm-z2z32}.

\subsection{Proof of Theorem~\ref{thm-main1}} \label{sec-pfthmn1}

To deduce Theorem~\ref{thm-main1} from Theorem~\ref{thm-main3}, we must take $S = W_N(\Z)$ and evaluate the integral over $\Z_p$ for each prime $p$. We do this as follows:

\begin{prop} \label{prop-specialvol}
  Fix a prime $p$. Then we have that
  $$\frac{1}{\on{Vol}(G_N(\Z_p))}\int_{w \in W_N^0(\Z_p)} |\lambda(w)|_pdw =  \prod_{i = 2}^N \frac{1}{1 - p^{-i}}.$$
\end{prop}
\begin{proof}
Our strategy is to partition $W_N^0(\Z_p)$ into the union over $\vec{a},\vec{b} \in \mathbb{N}^n$ of the preimage under projection $\pi \colon W_N^0(\Z_p) \to W_N^{\on{top}}(\Z_p)$ of the set $\mc{L}_{\vec{a},\vec{b}}(p)$ defined in \S\ref{sec-fundzits}. Taking $\phi$ to be the indicator function of $\pi^{-1}(\mc{L}_{\vec{a},\vec{b}}(p))$ and applying Proposition~\ref{prop-jac2} and Lemma~\ref{lem-ksiolajidebt} yields that
\begin{align} \label{eq-vollkfurreal}
& \int_{w \in \pi^{-1}(\mc{L}_{\vec{a},\vec{b}}(p))} |Q(w)|_pdw = \\
& \qquad\quad \frac{\on{Vol}\big((\on{SL}_n \times \on{SL}_{n+1})(\Z_p)\big)}{\big(\prod_{i = 1}^n p^{(n+1-i)a_i+ib_i}\big)^2} \int_{\substack{q \in \Z_p \\ \nu_p(q) = \sum_{i = 1}^n (n+1-i)a_i+ib_i}} \#\left(\frac{Q^{-1}(q) \cap \mc{L}_{\vec{a}, \vec{b}}(p)}{(\on{SL}_n \times \on{SL}_{n+1})(\Z_p)}\right) dq. \nonumber
\end{align}
Now, let $\nu_p$ denote the usual $p$-adic valuation. Proposition~\ref{prop-padicfund} along with the formula~\eqref{eq-formforq} for the $Q$-invariant implies that we have for all $q$ with $\nu_p(q) = \sum_{i = 1}^n (n+1-i)a_i+ib_i$ that
\begin{equation} \label{eq-innermass}
    \#\left(\frac{Q^{-1}(q) \cap \mc{L}_{\vec{a}, \vec{b}}(p)}{(\on{SL}_n \times \on{SL}_{n+1})(\Z_p)}\right) = \prod_{i = 1}^n p^{(n-i)a_i + ib_i},
\end{equation}
so substituting~\eqref{eq-innermass} along with the calculation of $\on{Vol}\big((\on{SL}_n \times \on{SL}_{n+1}(\Z_p)\big)$ from Lemma~\ref{lem-gvol} into the right-hand side of~\eqref{eq-vollkfurreal} yields
\begin{align} \label{eq-vollkwhatsit}
\int_{w \in \pi^{-1}(\mc{L}_{\vec{a},\vec{b}}(p))} |Q(w)|_pdw & =  \frac{(1-p^{-1})(1 - p^{-n-1}) \times \prod_{i = 2}^n ( 1- p^{-i})^2 \times \prod_{i = 1}^n p^{(n-i)a_i+ib_i}}{\big(\prod_{i = 1}^n p^{(n+1-i)a_i+ib_i}\big)^{3}}.
\end{align}
Summing~\eqref{eq-vollkwhatsit} over all $\vec{a},\vec{b} \in \mathbb{N}^n$ and using the calculation of $\on{Vol}(G_N(\Z_p))$ from Lemma~\ref{lem-gvol}, we have that
\begin{align} \label{eq-finalcountdowns}
\frac{1}{\on{Vol}(G_N(\Z_p))}\int_{w \in W_N^0(\Z_p)} |Q(w)|_pdw & = \frac{1}{\on{Vol}(G_N(\Z_p))}\sum_{\vec{a},\vec{b} \in \mathbb{N}^n} \int_{w \in \pi^{-1}(\mc{L}_{\vec{a},\vec{b}}(p))} |Q(w)|_pdw \nonumber \\
& = \prod_{i = 1}^n \sum_{\vec{a} \in \mathbb{N}^n} \frac{1}{p^{(2n+3-2i)a_i}} \times \sum_{\vec{b} \in \mathbb{N}^n} \frac{1}{p^{2ib_i}} \nonumber \\
& =  \prod_{i = 1}^n \frac{1}{1 - p^{-(N-(2i-2))}} \times \frac{1}{1-p^{-2i}} =  \prod_{i = 2}^N \frac{1}{1 - p^{-i}}, \nonumber
\end{align}
which is the desired result.
\end{proof}

Theorem~\ref{thm-main1} now follows by combining Proposition~\ref{prop-specialvol} with Theorem~\ref{thm-main3}, and by evaluating the resulting Euler product. 

\subsection{Proofs of Theorems~\ref{thm-z2z3} and~\ref{thm-z2z32}} \label{sec-pf2tors}

Let $R$ be a principal ideal domain. 
\begin{defn} \label{def-proj}
We say that (the $(\on{GL}_2 \times \on{SL}_3)(R)$-orbit of) a pair $(A,B) \in W_3(R)$ is \emph{projective} if the following property is satisfied. Write $\on{inv}(A,B) = \sum_{i = 0}^3 f_ix^{3-i}y^i$, and for each $k \in \{0,1,2\}$, let $C^{(k)}$ be the $3 \times 3$ matrix over $R$ defined as follows:
\begin{equation} \label{eq-defcks}
C^{(0)} = BA^{*}B, \quad C^{(1)} = B, \quad C^{(2)} = A,
\end{equation}
where $A^*$ denotes the adjugate matrix of $A$. 
Let $M \in \on{Mat}_{3 \times 6}(R)$ be the matrix whose $k^{\mathrm{th}}$ row consists of the entries of $C^{(k)}$ lying on or above the diagonal, written as a list in some order that is uniform over $k$. Then the pair $(A,B)$ is projective if and only if the greatest common divisor of the $3 \times 3$ minors of $M$ is equal to $1$.
\end{defn}

One can check from the above definition that: (1) projectivity is a $(\on{GL}_2 \times \on{SL}_3)(R)$-invariant condition; (2) projectivity over $\Z$ is equivalent to projectivity over $\Z_p$ for every prime $p$; (3) projectivity over $\Z_p$ is a mod-$p$ condition (i.e., whether or not a pair $(A,B)$ is projective is determined by the residue class of $(A,B)$ modulo $p$); and (4) any pair $(A,B) \in W_3^0(\F_p)$ with $Q(A,B) \neq 0$ is projective --- this last claim follows from Proposition 2.9, which implies that it suffices to verify the claim for the image of the section $\sigma_0$, which is easily done. In particular, the set of projective elements of $W_3(\Z)$ is a big family in $W_3(\Z)$.

The motivation to introduce the notion of projectivity is the following parametrization result, which relates $2$-torsion ideals of rings defined by binary cubic forms to projective reducible $\on{SL}_3(\Z)$-orbits on $W_3(\Z)$:

\begin{theorem} \label{thm-bharg}
  The elements of the group $\mc{I}(R_f)[2]$ are in natural bijection with the projective reducible $\on{SL}_3(\Z)$-orbits of pairs $(A,B) \in W_3(\Z)$ with $-\det(xA - yB) = f(x,y)$.
\end{theorem}
\begin{proof}
Consider the set $H_f$ of equivalence classes of pairs $(I,\delta)$, where $I$ is a fractional ideal of $R_f$ and $\delta \in K_f^\times$ are such that we have the containment $I^2 \subset (\delta)$ and equality of norms $\on{N}(I)^2 = \on{N}(\delta)$, and where two such pairs $(I_1, \delta_1)$ and $(I_2, \delta_2)$ are equivalent if there exists $\kappa \in K_f^\times$ such that $I_1 = \kappa I_2$ and $\delta_1 = \kappa^2\delta_2$. By~\cite[Theorem~5.7]{MR3187931}, the set $H_f$ is in natural bijection with the set of $\on{SL}_3(\Z)$-orbits of pairs $(A,B) \in W_3(\Z)$ with $-\det(xA - yB) = f(x,y)$; for an explicit construction of this bijection, see~\cite[\S2.2, p.~1007]{MR3782066}.

Now, we have an injection $\mc{I}(R_f)[2] \hookrightarrow H_f$, given by sending $I$ to the equivalence class of $(I,1)$; thus, we may regard $\mc{I}(R_f)[2]$ as a subset of $H_f$, and it suffices to determine the image of this subset under the bijection referenced above. This image was determined in~\cite[Lemma~16]{MR3369305} to be the set of projective reducible $\on{SL}_3(\Z)$-orbits of pairs $(A,B) \in W_3(\Z)$ with $-\det(xA - yB) = f(x,y)$. Note that a different but equivalent definition of projectivity is used in~\cite{MR3369305}---there, an orbit is projective if it corresponds to the equivalence class of a pair $(I,\delta)$ with $I$ invertible. The equivalence of the two definitions can be shown using the aforementioned explicit construction of the bijection (see~\cite[\S2.2,~p.~1007]{MR3782066}), from which it follows that $I$ is invertible if and only if it corresponds to a pair $(A,B)$ such that the gcd criterion in Definition~\ref{def-proj} is satisfied.
\end{proof}

By Theorem~\ref{thm-bharg}, proving Theorems~\ref{thm-z2z3} and~\ref{thm-z2z32} amounts to determining asymptotics for the number of projective reducible $\on{SL}_3(\Z)$-orbits on $W_3(\Z)$. An argument entirely analogous to the proof of Proposition~\ref{prop-sasymp} implies that these asymptotics are the same as the asymptotics for the number of projective $G_3(\Z)$-orbits on $W_3^0(\Z)$. By Theorem~\ref{thm-acceptcubic2}, Proposition~\ref{prop-jac}, and Lemma~\ref{lem-ksiolajidebt}, this amounts to evaluating a certain $p$-adic integral for each prime $p$, which we \mbox{do as follows:}
\begin{prop}
Fix a prime $p$. Then we have that
$$\frac{1}{\on{Vol}(G_3(\Z_p))}\int_{\substack{w \in W_3^0(\Z_p) \\ w\, \mathrm{is}\, \mathrm{proj.}}} |\lambda(w)|_pdw = 1+p^{-2}.$$
\end{prop}
\begin{proof}
Our strategy is to slice up the set of projective elements of $W_3^0(\Z_p)$ into level sets for the function $w \mapsto \nu_p(\lambda(w))$, to evaluate the integral on each level set, and to sum up the results. To this end, given an integer $k \geq 1$, let $L_k \defeq \{(A,B) \in W_3^0(\Z_p) : \nu_p(Q(A,B)) = k-1\}$, given \mbox{$a,b,c,d \in \F_p$, let}
$$S_{a,b,c,d} \defeq \{(A,B) \in W_3^0(\F_p) : A_{12} = a,\,A_{13} = b,\,B_{12} = c,\,B_{13} = d, \text{ and $(A,B)$ is proj.}\},$$
and given $m \in \Z/p^k\Z$, denote by $\ol{m}$ the mod-$p$ reduction of $m$. Then we have that
\begin{equation} \label{eq-vollkproj}
  \int_{\substack{w \in L_k \\ w\,\mathrm{is}\,\mathrm{proj}.}} |\lambda(w)|_pdw = p^{1-k} \times p^{-4k} \times \sum_{\substack{M \in \on{Mat}_{2 \times 2}(\Z/p^k\Z) \\ \nu_p(\det M) = k-1}} p^{-6} \times \#S_{\ol{M}_{11},\ol{M}_{12},\ol{M}_{21},\ol{M}_{22}}.
\end{equation}
We now determine the size of the set $S_{a,b,c,d}$ for each choice of $a,b,c,d \in \F_p$:
\begin{itemize}
\item Let $(A,B) \in W_3^0(\F_p)$. As mentioned above, if $\lambda(A,B) \neq 0$, then $(A,B)$ is projective. Thus, if $ad-bc \neq 0$, then $\#S_{a,b,c,d} = p^6$. 
\item Now suppose that $\lambda(A,B) = 0$. A calculation reveals that if $p$ divides all of $A_{12}$, $A_{13}$, $B_{12}$, and $B_{13}$, then $(A,B)$ is not projective. Thus, for $(A,B)$ to be projective, at least one of these four matrix entries must be a unit. We now fiber over these four matrix entries and determine the number of possibilities for the pair $(A,B)$ in each fiber. Fix four elements $a,b,c,d \in \F_p$ with $ad - bc = 0$ and $\{0\} \neq \{a,b,c,d\}$. We claim that $\#S_{a,b,c,d}$ is independent of the choice of $a,b,c,d$. Indeed, if $a',b',c',d' \in \F_p$ with $a'd' - b'c' = 0$ and $\{0\} \neq \{a',b',c',d'\}$, then there exists $\gamma \in H_2(\F_p)$ such that if we set $(A',B') \defeq \gamma \cdot (A,B)$, then $(A'_{12},A'_{13},B'_{12},B'_{13}) = (a,b,c,d)$. Thus, $\gamma$ induces a bijection between $S_{a,b,c,d}$ and $S_{a',b',c',d'}$, so it suffices to compute $\#S_{0,1,0,0}$. A calculation reveals that a pair $(A,B)$ with $(A_{12}, A_{13}, B_{12}, B_{13}) = (0,1,0,0)$ is projective if and only if $B_{22}B_{33} - B_{23}^2 \neq 0$. 
Using this characterization, it is easy to check that $\#S_{0,1,0,0} = p^5(p-1)$.
\end{itemize}
Substituting the formulas for $\#S_{a,b,c,d}$ obtained above into the right-hand side of~\eqref{eq-vollkproj}, we find that
\begin{equation} \label{eq-vollkproj2}
 \int_{\substack{w \in L_k \\ w\,\mathrm{is}\,\mathrm{proj}.}} |\lambda(w)|_pdw =  p^{1 - 5k} \times \begin{cases} c(k), & \text{if $k = 1$}, \\ 
 p^{-1}(p-1) \times c(k), & \text{if $k \geq 2$.} \end{cases}
\end{equation}
where $c(k) \defeq \#\{M \in \on{Mat}_{2 \times 2}(\Z/p^k\Z) : \nu_p(\det M) = k-1 \text{ and } M \not\equiv 0 \pmod p\}$.  Let $c'(k) \defeq \#\{M \in \on{Mat}_{2 \times 2}(\Z/p^k\Z) : \nu_p(\det M) = k-1\}$. Then we have 
\begin{equation} \label{eq-cppinc}
c(k) = c'(k) - p^{4}c'(k-2),
\end{equation}
where for convenience we set $c'(m) = 0$ if $m \leq 0$. The next lemma computes $c'(k)$:
  \begin{lemma} \label{lem-matcount}
    Let $k \geq 1$ be as above. Then $c'(k) = p^{2k-1}(p^2-1)(p^k-1)$.
  \end{lemma}
  \begin{proof}[Proof of Lemma~\ref{lem-matcount}]
    For $d \in \Z/p^k \Z$, let $M(d) \defeq \{M \in \on{Mat}_{2 \times 2}(\Z/p^k\Z) : M_{11}M_{22} = M_{12}M_{21} = d\}$, and let $N(d) \defeq \{(a,b) \in (\Z/p^k \Z)^2 : ab = d\}$. Then $M(d) = N(d)^2$, and so if we set $c''(k) \defeq \#\{M \in \on{Mat}_{2 \times 2}(\Z/p^k\Z) : \det M = 0\}$ of matrices over $\Z/p^k\Z$ with determinant $0$, then we have
    \begin{equation} \label{eq-mton}
    c''(k) = \sum_{d \in \Z/p^k \Z} M(d) = \sum_{d \in \Z/p^k \Z} N(d)^2.
    \end{equation}
    Let $\phi$ denote Euler's totient function. First suppose $\nu_p(d) < k$, and fix a number $j \in \{0, \dots, \nu_p(d)\}$. If $d$ factors as $d = ab$ where $\nu_p(a) =j$, then we have that there are $\phi(p^{k-j})$ choices for $a$ and $\phi(p^{k-\nu_p(d)+j})\phi(p^{k-\nu_p(d)})^{-1}$ choices for $b$. Summing over all $j$, \mbox{we find that}
    \begin{equation} \label{eq-whatsnot0}
      N(d) = \sum_{j = 0}^{\nu_p(d)} \phi(p^{k-\nu_p(d)})^{-1}\phi(p^{k-j})\phi(p^{k-\nu_p(d)+j}) \quad \text{if $\nu_p(d) < k$.}
    \end{equation}
    Now suppose $d \equiv 0 \pmod{p^k}$. Suppose $d$ factors as $d = ab$, let $i = \nu_p(a)$ if $\nu_p(a) < k$ and $i = k$ otherwise. Then there are $\phi(p^{k-i})$ choices for $a$, and for each $j \in \{k-i, \dots, k\}$, there are $\phi(p^{k-j})$ choices for $b$. Summing over all $i$ and $j$, we find that
    \begin{equation} \label{eq-whatsn0}
      N(0) = \sum_{i = 0}^{k} \sum_{j  =k-i}^{k} \phi(p^{k-i})\phi(p^{k-j}).
    \end{equation}
    Substituting the results of~\eqref{eq-whatsnot0} and~\eqref{eq-whatsn0} into~\eqref{eq-mton} and evaluating the sum, we deduce that
    \begin{equation} \label{eq-whatsc'}
    c''(k) = p^{2k-1}(p^k(p+1)-1).
    \end{equation}
    Finally, note that we have $c'(k) = p^4 c''(k-1) - c''(k)$. Substituting in the formula~\eqref{eq-whatsc'} for $c''(k)$ yields the lemma.
  \end{proof}
  Substituting the formula for $c'(k)$ given by Lemma~\ref{lem-matcount} into~\eqref{eq-cppinc} yields 
$$c(k) = c'(k) - p^{4}c'(k-2) =  \begin{cases} p(p-1)(p^2-1), & \text{if $k = 1$,} \\ 
p^{3k-3}(p^2-1)^2, & \text{if $k \geq 2$.} \end{cases}$$
Substituting this expression for $c(k)$ into the right-hand side of~\eqref{eq-vollkproj2}, summing up over all positive integers $k$, and dividing by the volume $\on{Vol}(G_3(\Z_p))$ as given in Lemma~\ref{lem-gvol} yields the proposition.
\end{proof}

\section*{Acknowledgments}

\noindent It is a pleasure to thank Manjul Bhargava, Andrew Granville, Arul Shankar, Artane Siad, Ila Varma, and Melanie Matchett Wood for several helpful discussions. We are also grateful to the anonymous referee for numerous insightful comments and corrections. This material is based in part upon work supported by the National Science Foundation, under the Graduate Research Fellowship as well as Award No.~2202839.

	\bibliographystyle{plain}
	\bibliography{references}

\begin{thebibliography}{10}

\bibitem{bhargthesis}
{M}. Bhargava.
\newblock {\em Higher composition laws}.
\newblock ProQuest LLC, Ann Arbor, MI, 2001.
\newblock Thesis (Ph.D.)--Princeton University.

\bibitem{MR2081442}
{M}. Bhargava.
\newblock Higher composition laws. {II}. {O}n cubic analogues of {G}auss
  composition.
\newblock {\em Ann. of Math. (2)}, 159(2):865--886, 2004.

\bibitem{MR2113024}
{M}. Bhargava.
\newblock Higher composition laws. {III}. {T}he parametrization of quartic
  rings.
\newblock {\em Ann. of Math. (2)}, 159(3):1329--1360, 2004.

\bibitem{MR2183288}
{M}. Bhargava.
\newblock The density of discriminants of quartic rings and fields.
\newblock {\em Ann. of Math. (2)}, 162(2):1031--1063, 2005.

\bibitem{MR2373152}
{M}. Bhargava.
\newblock Higher composition laws. {IV}. {T}he parametrization of quintic
  rings.
\newblock {\em Ann. of Math. (2)}, 167(1):53--94, 2008.

\bibitem{MR2745272}
{M}. Bhargava.
\newblock The density of discriminants of quintic rings and fields.
\newblock {\em Ann. of Math. (2)}, 172(3):1559--1591, 2010.

\bibitem{thesource}
{M}. Bhargava.
\newblock Most hyperelliptic curves over $\mathbb{Q}$ have no rational points.
\newblock {\em arXiv preprint 1308.0395}, 2013.

\bibitem{MR3600041}
M.~Bhargava, B.~H. Gross, and X.~Wang.
\newblock A positive proportion of locally soluble hyperelliptic curves over
  {$\Bbb Q$} have no point over any odd degree extension.
\newblock {\em J. Amer. Math. Soc.}, 30(2):451--493, 2017.
\newblock With an appendix by Tim Dokchitser and Vladimir Dokchitser.

\bibitem{BSHpreprint}
{M}. Bhargava, {J}. Hanke, and {A}. {S}hankar.
\newblock The mean number of $2$-torsion elements in class groups of
  $n$-monogenized cubic fields.
\newblock {\em arXiv preprint 2010.15744}, 2020.

\bibitem{MR3272925}
{M}. Bhargava and {A}. Shankar.
\newblock Binary quartic forms having bounded invariants, and the boundedness
  of the average rank of elliptic curves.
\newblock {\em Ann. of Math. (2)}, 181(1):191--242, 2015.

\bibitem{MR3090184}
{M}. Bhargava, {A}. Shankar, and {J}. Tsimerman.
\newblock On the {D}avenport-{H}eilbronn theorems and second order terms.
\newblock {\em Invent. Math.}, 193(2):439--499, 2013.

\bibitem{sqfrval2}
{M}. Bhargava, {A}. Shankar, and {X}. Wang.
\newblock Squarefree values of polynomial discriminants {II}.
\newblock {\em arXiv preprint 2207.05592}, 2022.

\bibitem{MR3369305}
{M}. Bhargava and {I}. Varma.
\newblock On the mean number of 2-torsion elements in the class groups, narrow
  class groups, and ideal groups of cubic orders and fields.
\newblock {\em Duke Math. J.}, 164(10):1911--1933, 2015.

\bibitem{MR3471250}
{M}. Bhargava and {I}. Varma.
\newblock The mean number of 3-torsion elements in the class groups and ideal
  groups of quadratic orders.
\newblock {\em Proc. Lond. Math. Soc. (3)}, 112(2):235--266, 2016.

\bibitem{MR756082}
H.~Cohen and H.~W. Lenstra, Jr.
\newblock Heuristics on class groups of number fields.
\newblock In {\em Number theory, {N}oordwijkerhout 1983 ({N}oordwijkerhout,
  1983)}, volume 1068 of {\em Lecture Notes in Math.}, pages 33--62. Springer,
  Berlin, 1984.

\bibitem{MR866103}
H.~Cohen and J.~Martinet.
\newblock Class groups of number fields: numerical heuristics.
\newblock {\em Math. Comp.}, 48(177):123--137, 1987.

\bibitem{MR43822}
H.~Davenport.
\newblock On the class-number of binary cubic forms. {I}.
\newblock {\em J. London Math. Soc.}, 26:183--192, 1951.

\bibitem{MR1574296}
H.~Davenport.
\newblock Corrigendum: {O}n the {C}lass-{N}umber of {B}inary {C}ubic {F}orms
  ({I}).
\newblock {\em J. London Math. Soc.}, 27(4):512, 1952.

\bibitem{MR491593}
H.~Davenport and H.~Heilbronn.
\newblock On the density of discriminants of cubic fields. {II}.
\newblock {\em Proc. Roy. Soc. London Ser. A}, 322(1551):405--420, 1971.

\bibitem{MR0160744}
{B}.~{N}. Delone and {D}.~{K}. Faddeev.
\newblock {\em The theory of irrationalities of the third degree}.
\newblock Translations of Mathematical Monographs, Vol. 10. American
  Mathematical Society, Providence, R.I., 1964.

\bibitem{MR3782066}
{W}. Ho, {A}. Shankar, and {I}. Varma.
\newblock Odd degree number fields with odd class number.
\newblock {\em Duke Math. J.}, 167(5):995--1047, 2018.

\bibitem{MR2778658}
{G}. Malle.
\newblock On the distribution of class groups of number fields.
\newblock {\em Experiment. Math.}, 19(4):465--474, 2010.

\bibitem{MR430336}
M.~Sato and T.~Kimura.
\newblock A classification of irreducible prehomogeneous vector spaces and
  their relative invariants.
\newblock {\em Nagoya Math. J.}, 65:1--155, 1977.

\bibitem{MR516601}
{G}.~{W.} Schwarz.
\newblock Representations of simple {L}ie groups with a free module of
  covariants.
\newblock {\em Invent. Math.}, 50(1):1--12, 1978/79.

\bibitem{cuspy}
{A}. {S}hankar, {A}. {S}iad, {A}. Swaminathan, and {I}. {V}arma.
\newblock Geometry-of-numbers methods in the cusp.
\newblock {\em arXiv preprint arXiv:2110.09466}, 2021.

\bibitem{Siadthesis1}
{A}. Siad.
\newblock Monogenic fields with odd class number part {I}: odd degree.
\newblock {\em arXiv preprint 2011.08834}, 2020.

\bibitem{Siadthesis2}
{A}. Siad.
\newblock Monogenic fields with odd class number part {II}: even degree.
\newblock {\em arXiv preprint 2011.08842}, 2020.

\bibitem{swathesis}
{A}. Swaminathan.
\newblock {\em 2-{S}elmer {G}roups, 2-{C}lass {G}roups, and the {A}rithmetic of
  {B}inary {F}orms}.
\newblock ProQuest LLC, Ann Arbor, MI, 2022.
\newblock Thesis (Ph.D.)--Princeton University.

\bibitem{Swpreprint}
{A}. Swaminathan.
\newblock A new parametrization for ideal classes in rings defined by binary
  forms, and applications.
\newblock {\em J. Reine Angew. Math.}, 798:143--191, 2023.

\bibitem{MR3187931}
{M}.~{M}. Wood.
\newblock Parametrization of ideal classes in rings associated to binary forms.
\newblock {\em J. Reine Angew. Math.}, 689:169--199, 2014.

\end{thebibliography}

\end{document}